\documentclass[reqno,12 pt]{amsart}
\usepackage[norelsize,ruled,vlined,boxed,commentsnumbered]{algorithm2e}
\usepackage{babel}
\usepackage[T1]{fontenc}
\usepackage[latin9]{inputenc}
\usepackage{float}
\usepackage{amsfonts,amscd,mathrsfs}
\usepackage{amsthm}
\usepackage{amssymb}
\usepackage{lineno}[switch, displaymath, mathlines]
\usepackage{latexsym}
\usepackage{wasysym}
\usepackage{subcaption}
\usepackage{enumerate}
\usepackage{url}
\usepackage{cite}
\usepackage{amsmath}
\usepackage{verbatim}
\usepackage{graphicx}
\usepackage{epsf}
\usepackage{geometry}
\usepackage{color}

\usepackage{hyperref}
\definecolor{dark-blue}{rgb}{0,0,0.6}
\definecolor{Purple}{rgb}{0.2,0,0.25}
\hypersetup{breaklinks=true,
pagecolor=white,
colorlinks=true,
citecolor=Purple,
filecolor=black,%
linkcolor=blue,%
urlcolor=blue
}

\theoremstyle{plain}
\newtheorem{thm}{\protect\theoremname}
\theoremstyle{definition}
\newtheorem{defn}[thm]{\protect\definitionname}
\theoremstyle{plain}
\newtheorem{lem}[thm]{\protect\lemmaname}
\theoremstyle{plain}
\newtheorem{prop}[thm]{\protect\propositionname}
\theoremstyle{remark}
\newtheorem{rem}[thm]{\protect\remarkname}
\theoremstyle{definition}
\newtheorem{example}[thm]{\protect\examplename}

\makeatother

\providecommand{\definitionname}{Definition}
\providecommand{\examplename}{Example}
\providecommand{\lemmaname}{Lemma}
\providecommand{\propositionname}{Proposition}
\providecommand{\remarkname}{Remark}
\providecommand{\theoremname}{Theorem}

\numberwithin{equation}{section}

\newcommand{\bref}[1]{\textbf{\ref{#1}}} 
\newcommand{\beqref}[1]{\textbf{(\ref{#1})}}

\date{Date: June 5, 2024.}



\begin{document}
\title[The alternating HLWB algorithm for finding the BAP]{The alternating simultaneous Halpern--Lions--Wittmann--Bauschke
algorithm for finding the best approximation pair for two disjoint
intersections of convex sets}

\author{Yair Censor}

\address{Department of Mathematics, University of Haifa, Mt. Carmel, 3498838 Haifa,  Israel.}
\email{(Yair Censor)  yair@math.haifa.ac.il}

\author{Rafiq Mansour}

\address{Department of Mathematics, University of Haifa, Mt. Carmel, 3498838 Haifa,  Israel. Current address: Mawares E-Thahab 2, Tira, P. O. B. 562, zip code: 44915, Israel.}
\email{(Rafiq Mansour) intogral@gmail.com} 

\author{Daniel Reem}

\address{The Center for Mathematics and Scientific Computation (CMSC), University of Haifa, Mt. Carmel, 3498838 Haifa, Israel.} 
\email { (Daniel Reem) dream@math.haifa.ac.il} 

\maketitle

\begin{abstract} 
Given two nonempty and disjoint intersections of closed and convex
subsets, we look for a best approximation pair relative to them, i.e.,
a pair of points, one in each intersection, attaining the minimum
distance between the disjoint intersections. We propose an iterative
process based on projections onto the subsets which generate the intersections. The process is inspired by the Halpern-Lions-Wittmann-Bauschke algorithm and the classical alternating process of Cheney and Goldstein, and its advantage is that there is no need to project onto the intersections themselves, a task which can be rather demanding. We prove that under certain conditions the two interlaced subsequences converge to a best approximation pair. These conditions hold, in particular, when the
space is Euclidean and the subsets which generate the intersections
are compact and strictly convex. Our result extends the one of  Aharoni,
Censor and Jiang {[}``Finding a best approximation pair of points
for two polyhedra'', \textit{Computational Optimization and Applications}
71 (2018), 509--523{]} who considered the case of finite-dimensional
polyhedra. 
\end{abstract}

\noindent \tableofcontents{}

\section{Introduction}

\subsection{Background:}

We consider the problem of finding a best approximation pair relative
to two closed and convex disjoint sets $A$ and $B$, namely a pair
$\left(a,b\right)\in A\times B$ satisfying $\left\Vert a-b\right\Vert =\inf\left\{ \left\Vert x-y\right\Vert \mid x\in A,\:y\in B\right\} $.
This problem goes back to the classical 1959 work of Cheney and Goldstein
\cite{cheney1959proximity} which employs proximity maps (i.e., metric
projections, Euclidean nearest point projections) and is at the heart
of what is nowadays called ``finding best approximation pairs relative
to two sets''.

We consider the important situation where each of the two sets is
a nonempty intersection of a finite family of compact and convex sets
and the space is Euclidean (i.e., a finite-dimensional real Hilbert space).
The practical importance of this situation stems from its relevance
to real-world situations wherein the feasibility-seeking modelling
is used and there are two disjoint constraints sets. One set represents
``hard'' constraints, i.e., constraints that must be met, while
the other set represents ``soft'' constraints which should be observed
as much as possible. Under such circumstances, the desire to find
a point in the hard constraints intersection set that will be closest
to the intersection set of soft constraints leads to the problem of
finding a best approximation pair relative to the two sets which are intersections
of constraints sets: see, e.g., \cite{GoldburgMarks1985jour} and
\cite{YoulaVelasco1986jour}, and also the more general approach
\cite{hard-soft-1999}, for applications in signal processing.

One way to approach the best approximation pair problem is to project
alternatingly onto each of the two intersections and then to take
the limit, as done, for instance, in \cite[Section 3]{bauschke1993convergence},
\cite[Theorem 4.8]{bauschke1994dykstra}, \cite[Theorem 4]{cheney1959proximity},
\cite[Theorem 4.1]{KopeckaReich2004jour}, \cite[Theorem 1.4]{kopecka2012note}
and \cite[Theorem 1]{YoulaVelasco1986jour}, or to apply iteratively
other kinds of operators (related to orthogonal projections), as done
in \cite[Theorem 3.13]{bauschke2004finding}, \cite[p. 656]{GoldburgMarks1985jour}
and \cite[Corollary 2.8]{luke2008finding}; see also the review \cite{censor2018algorithms}.
While this approach is satisfying from the mathematical point of view,
it can be rather problematic from the computational point of view
because applying a metric projection operator to the intersection
of finitely-many sets is commonly a non-trivial problem on its own
which might be solved only approximately. (We note, parenthetically,
that there exist, of course, iterative methods which can find asymptotically,
and sometimes exactly, a projection of a point onto the intersection
of a finite family of closed and convex subsets: see, for instance,
\cite{BauschkeBorwein1996jour,BC17,Ceg-book,CensorCegielski2015jour,CensorZenios1997book}
and some of the references therein.)

We approach the best approximation pair problem from a constructive
algorithmic point of view. A major advantage of our approach is that
it eliminates the need to project onto  each of the two intersection
sets, and replaces these computationally demanding projections by
a certain weighted sum of projections onto each of the members which
induce  these intersection sets. The method that we employ is inspired
by  the HLWB algorithm of \cite[Corollary 30.2]{BC17}, which we
apply alternatingly to the two intersections (the letters HLWB are
acronym of the names of the authors of the corresponding papers: Halpern
in \cite{halpern1967fixed}, Lions in \cite{lions1977approximation},
Wittmann in \cite{wittmann1992approximation}, and Bauschke in \cite{bauschke1996approximation};
the acronym HLWB was dubbed in \cite{censor2006computational}).

More precisely, the iterative process that we consider is divided
into sweeps, where in the odd numbered sweeps we project successively
onto a collection of compact and convex subsets $A_{i}$ defining
$A:=\bigcap_{i=1}^{I}A_{i}\neq\emptyset,\,I\in\mathbb{N}$, and construct
from all of these projections certain weighted sums, and in even numbered
sweeps we act similarly but with the collection of compact and convex
subsets $B_{j}$ defining $B:=\bigcap_{j=1}^{J}B_{j}\neq\emptyset,\,J\in\mathbb{N}$,
where $A\bigcap B=\emptyset$. An important component in the method
is that the number of successive weighted sums increases from sweep
to sweep. Under the assumption that there exists a unique best approximation
pair, we are able to show that our algorithmic scheme converges to
this pair. This assumption, of uniqueness of the best approximation
pair, is satisfied, in particular, when all the sets $A_{i},\,B_{j},\,i\in\left\{ 1,2,\ldots,I\right\} $
and $j\in\left\{ 1,2,\ldots,J\right\} $ are strictly convex, as we
show in Proposition \bref{prop:Proposition 6} below.

Our work is motivated by the work of Aharoni et al. \cite{aharoni2018finding},
in which the authors present an algorithm which looks for a best approximation
pair relative to given two disjoint convex polyhedra. In that work
they propose a process based on projections onto the half-spaces defining
the two polyhedra, and this process is essentially a sequential (that
is, non-simultaneous) alternating version of the HLWB algorithm; in
contrast, the method that we employ is a simultaneous version of the
HLWB algorithm. While the given subsets in their case are essentially
linear (affine), in our case they are essentially nonlinear (non-affine);
on the other hand, the subsets in \cite{aharoni2018finding} are
not necessarily bounded as in our case, and no uniqueness assumption
on the set of best approximation pairs is imposed in \cite{aharoni2018finding}
like is done here; hence our work extends, but does not generalize,
their work.

The proof of our main convergence result (Theorem \bref{thm:Main}
below) is partly inspired by the proof of convergence of the main
result in \cite{aharoni2018finding}, namely \cite[Theorem 1]{aharoni2018finding},
but important differences exist because the settings are different.
In particular, along the way we present and use a simple but useful
generalization of the celebrated Dini's theorem for uniform convergence
(Proposition \bref{prop:generalized-Dini's-Theorem} below).

We note that also \cite{bauschke2022finding} is motivated by \cite{aharoni2018finding}
and extends it to closed and convex subsets which are not just polyhedra.
The authors of \cite{bauschke2022finding} study the Douglas--Rachford
algorithm, a dual-based proximal method, a proximal distance algorithm,
and a stochastic subgradient descent approach. The presentation in
\cite{bauschke2022finding} is based on reformulating the problem
into a minimization problem and applying various approaches to it,
theoretically and experimentally, without convergence results (see
also \cite{Dax2006jour} for a dual reformulation of the best approximation
problem as a maximization problem, with some explicit examples), whereas
our present work maintains the original problem formulation and employs
tools from fixed point theory. In connection with \cite{bauschke2022finding}
it is also worth mentioning \cite[Remark 3.1]{bauschke2022finding}
which says that it is ``an interesting topic for further
research'' to extend the algorithm presented in \cite{aharoni2018finding}
``to the case when the sets underlying the intersections are not halfspaces''.
This is exactly what we do in this work, at least under the assumption
that the best approximation pair is unique. 

\subsection{Paper layout:}

The paper is organized as follows. In Section \bref{sec:Preliminaries}
we present various definitions and other relevant preliminaries. In Section \bref{sec:Sufficient condition}
we present a sufficient condition which ensures that there is a unique
best approximation pair. In Section \bref{sec:Features} we define
the Simultaneous-HLWB (S-HLWB) operator and introduce some of its
features. In Section \bref{sec:The method} we present our new method
for finding a best approximation pair in a finite-dimensional Hilbert
space. In Section \bref{sec:Convergence} we prove our main result,
namely the convergence of the alternating S-HLWB algorithmic sequence.
We conclude in Section \bref{sec:Concluding-remarks} with a few remarks
and lines for further investigation. 

\global\long\def\labelenumi{\roman{enumi}.}%

\section{Preliminaries\label{sec:Preliminaries}}

For the reader's convenience we include in this section some properties
of operators in Hilbert space that will be used in the sequel. Although
our setting is a Euclidean space, many of the definitions and results
mentioned in this section can be generalized to infinite-dimensional
spaces. We use the excellent books of Bauschke and Combettes \cite{BC17}
and of Cegielski \cite{Ceg-book} as our desk-copy in which all the
results of this section can be found. Let $\mathcal{H}$ be a real
Hilbert space with inner product $\left\langle \cdot,\cdot\right\rangle $
and induced norm $\parallel\cdot\parallel$, and let $X\subseteq\mathcal{H}$
be a nonempty, closed and convex subset. Let $C$ be a nonempty closed
and convex subset of $\mathcal{H}$, let $x\in\mathcal{H}$, and let
$c\in C$. Denote the distance from $x$ to $C$ by $d\left(x,C\right):=\textup{inf}_{c\in C}\Vert x-c\Vert$. It is well-known that the infimum is attained at a unique point called
the \textbf{projection (or the orthogonal projection) of $x$ onto
$C$,} and it is denoted by $P_{C}(x)$. Id stands below for the identity
operator and  $ri\,C$ stands for the relative interior of a set $C$.
We recall that $dist\left(A,B\right):=\inf\left\{ \left\Vert u-v\right\Vert \mid u\in A,v\in B\right\} $
whenever $A$ and $B$ are nonempty subsets of $X$, and we say that
$dist\left(A,B\right)$ is attained whenever there are $a\in A$ and
$b\in B$ such that $\left\Vert a-b\right\Vert =dist\left(A,B\right)$;
in this case we call $\left(a,b\right)$ a \textbf{best approximation
pair relative to $A$ and $B$}. It is straightforward to check that
in this case, if in addition $A$ and $B$ are closed and convex,
then $a=P_{A}b$ and $b=P_{B}a$.

Let $r$ be a nonengative number. The closed ball $B\left[x,r\right]$
is defined by 
\begin{equation}
B\left[x,r\right]:=\left\{ y\in\mathcal{H}\mid\left\Vert x-y\right\Vert \leq r\right\} ,
\end{equation}
for each $x\in\mathcal{H}$.

It is well-known that closed balls in a Euclidean space are compact,
see, e.g., \cite[Fact 2.33]{BC17}. For each $m\in\mathbb{N}$ we
denote by $\Delta_{m}$ the $(m-1)$-dimensional unit simplex, that
is $\vartriangle_{m}:=\{u=(u_i)_{i=1}^m\in\mathbb{R}^{m}\mid u_i\geq 0\,\,\textnormal{for all}\,\,i\,\, \textnormal{and} \sum_{i=1}^{m}u_{i}=1\}$. 
A vector $u\in\Delta_{m}$ is called a weight vector,
and a vector $u=(u)_{i=1}^{m}\in ri\,\Delta_{m}$ is called a positive
weight vector (that is, $\sum_{i=1}^{m}u_{i}=1$ and $u_{i}>0$ for
all $i\in\{1,\ldots,m\}$).
\begin{defn}
\label{def:NE+contraction}Let  $T:X\rightarrow\mathcal{H}$. Then
\end{defn}

\begin{enumerate}
\item $T$ is called \emph{Nonexpansive }(NE) if $\Vert T(x)-T(y)\Vert\leq\Vert x-y\Vert$
for all $x,y\in X$. 
\item $T$ is \textit{\emph{called}}\emph{ averaged}\textit{\emph{ if there
are $\lambda\in(0,1)$ and an NE operator }}$S:X\rightarrow\mathcal{H}$
\textit{\emph{such that $T=\lambda I+(1-\lambda)S$}}. 
\end{enumerate}
We have the following. 
\begin{lem}
\label{lem:comp of contra is a contra}\cite[Lemma 2.1.12]{Ceg-book}
Given $m\in\mathbb{N},$ let $S_{i}:X\rightarrow X$, be NE  for all
$i\in\left\{ 1,2,\ldots,m\right\} $ and let $w\in\Delta_{m}$. Then: 
\begin{enumerate}
\item The convex combination $S:=\sum_{i=1}^{m}w_{i}S_{i}$ is NE. 
\item The composition $S:=S_{m}S_{m-1}\cdots S_{1}$ is NE. 
\end{enumerate}
\end{lem}

\begin{prop}
\label{prop:P_C=00003D00003DNE} 
\begin{enumerate}
\item The orthogonal projection $P_{C}:\mathcal{H}\to\mathcal{H}$ onto
a nonempty, closed and convex subset $C$ of the space is NE, and
its fixed point set is $C$ itself. 
\item If for some pair $(x,y)\in\mathcal{H}^{2}$ equality holds in the
NE inequality related to $P_{C}$, that is, if $\|P_{C}x-P_{C}y\|=\|x-y\|$,
then $\|x-P_{C}x\|=\|y-P_{C}y\|$. 
\end{enumerate}
\end{prop}

\begin{proof}
For the proof of the first part see \cite[Theorem 2.2.21 (i) and (iv)]{Ceg-book}.
The proof of the second part (and, actually, of the first part as
well), can be found in \cite[Theorem 3]{cheney1959proximity}.
\end{proof}
\begin{defn}
\label{def:steering parameter seq.} A \textbf{steering parameters
sequence} $\left\{ \tau_{k}\right\} _{k=0}^{\infty}$ is a real sequence
in $\left(0,1\right)$ such that $\tau_{k}\rightarrow0$, $\sum_{k=0}^{\infty}\tau_{k}=+\infty$,
and $\sum_{k=0}^{\infty}\left|\tau_{k+1}-\tau_{k}\right|<+\infty$. 
\end{defn}

\begin{rem}
\label{rem:An-example-of-Steering-Parameter-Seq} An example of a
steering parameter sequence $\left\{ \tau_{k}\right\} _{k=1}^{\infty}$
is $\tau_{k}:=\dfrac{1}{k}$ for all $k\in\mathbb{N}$. 
\end{rem}

The following theorem is the corollary \cite[Corollary 30.2]{BC17}
with only some symbols changed to make it agree with the notations
that we use. The iterative process of \beqref{eq:simult-HLWB} represents
the \textbf{simultaneous HLWB (S-HLWB) algorithm}. 
\begin{thm}
\label{cor:simultaneous HLWB}Let $X$ be a nonempty closed convex
subset of $\mathcal{H}$, let $\left\{ T_{\ell}\right\} _{\ell=1}^{L},\,L\in\mathbb{N}$
be a finite family of  NE operators from $X$ to $X$ such that  $C:=\bigcap_{\ell=1}^{L}\textnormal{Fix}T_{\ell}\neq\emptyset$,
let $w=(w_{\ell})_{\ell=1}^{L}$ be a positive weight vector and let
$x\in X$ (called an anchor point). Let $\left\{ \tau_{k}\right\} _{k=0}^{\infty}$
be a sequence of steering parameters, let $x^{0}\in X$, and set 
\begin{equation}
\begin{split}\left(\forall k\geq0\right)\; & x^{k+1}:=\tau_{k}x+\left(1-\tau_{k}\right)\end{split}
\sum_{\ell=1}^{L}w_{\ell}T_{\ell}x^{k}.\label{eq:simult-HLWB}
\end{equation}
Then $\underset{k\rightarrow\infty}{\textnormal{lim}}x^{k}=P_{C}x$. 
\end{thm}

The theorem below was proved by Cheney and Goldstein in \cite[Theorem 2]{cheney1959proximity}.
See also \cite[Theorem 3]{GoldburgMarks1985jour} and \cite[Theorem 3]{YoulaVelasco1986jour}
for different proofs in the case where one of the sets is bounded
(\cite[Theorem 3]{GoldburgMarks1985jour} proves one direction, while
\cite[Theorem 3]{YoulaVelasco1986jour} proves  both directions in
the case where the Hilbert space is complex). 
\begin{thm}
\label{thm:cheney and goldstein Thm 2} Let $K_{1}$ and $K_{2}$
be two nonempty, closed and convex sets in a real Hilbert space. Let
$P_{i}$ denote the orthogonal projection onto $K_{i},\:i\in\left\{ 1,2\right\} $.
Then $x$ is a fixed point of $P_{1}P_{2}$ if and only if $x$ is
a point of $K_{1}$ nearest $K_{2}$. Moreover, in this case $\|x-P_{2}x\|=dist(K_{1},K_{2})$. 
\end{thm}

The following result is also due to Cheney and Goldstein in \cite[Theorem 4]{cheney1959proximity}. 
\begin{thm}
\label{thm:Cheney and Goldstein Thm 4} Let $K_{1}$ and $K_{2}$
be two nonempty, closed and convex sets in a real Hilbert space. Let
$P_{i}$ be the orthogonal projection onto $K_{i}$, $i\in\left\{ 1,2\right\} $
and let $Q:=P_{1}P_{2}$. If either one of the sets $K_{1}$ and $K_{2}$
is compact, or one of these sets is finite-dimensional and the distance
between these sets is attained, then for each $x$ in the space the
sequence $\left\{ Q^{k}\left(x\right)\right\} _{k=1}^{\infty}$ converges
to a fixed point of $Q$. 
\end{thm}

The next propositions describe the fixed point set of weighted sums and compositions of certain operators. 
   
\begin{prop}
\label{prop:Fix(combi)=00003D00003Dinter(Fix)}\cite[Theorem 2.1.14]{Ceg-book}
Given $m\in\mathbb{N},$ suppose that $U_{i}:X\rightarrow\mathcal{H},$
$i\in\left\{ 1,2,\ldots,m\right\} $ are  NE operators with a common
fixed point and let $U:=\sum_{i=1}^{m}w_{i}U_{i}$, where $w=(w_{i})_{i=1}^{m}$is
a positive weight vector.  Then 
\end{prop}

\noindent \textit{
\begin{equation}
\textnormal{Fix}U=\bigcap_{i=1}^{m}\textnormal{Fix}U_{i}.
\end{equation}
}

\begin{prop}
\label{prop:Fix(compo)=00003D00003Dinter(Fix)} \cite[Corollary 4.51]{BC17}
Given $m\in\mathbb{N},$ suppose that $T_{i}:\mathcal{H}\rightarrow\mathcal{H},$
$i\in\left\{ 1,2,\ldots,m\right\} $ are averaged  NE operators with a common
fixed point. If $T:=T_{m}T_{m-1}\cdots T_{1}$, then 
\begin{equation}
\textup{Fix}T=\bigcap_{i=1}^{m}\mathrm{Fix}T_{i}.
\end{equation}
\end{prop}

We now recall two basic notions of convergence. 
\begin{defn}
\label{def:uniformly convergence} Let $(X,d_{X})$ and $(Y,d_{Y})$
be two metric spaces and let $\left\{ f_{k}\right\} _{k=1}^{\infty}$,
$f_{k}:X\rightarrow Y$ be a sequence of functions. We say that $\left\{ f_{k}\right\} _{k=1}^{\infty}$

(i) converges pointwise to the function $f:X\rightarrow Y$ if for
every $\varepsilon>0$ and every $x\in X$ there exists a positive
integer $N_{\varepsilon}$ such that for all $k>N_{\varepsilon}$,
\begin{equation}
d_{Y}(f_{k}(x),f(x))<\varepsilon.
\end{equation}
(ii) converges uniformly to the function $f:X\rightarrow Y$ if for
every $\varepsilon>0$, there exists a positive integer $N_{\varepsilon}$
such that for all $k>N_{\varepsilon}$ and all $x\in X$, 
\begin{equation}
d_{Y}(f_{k}(x),f(x))<\varepsilon.
\end{equation}
\end{defn}

\section{A sufficient condition for the existence of a unique best approximation
pair\label{sec:Sufficient condition}}

In this section, we present a sufficient condition which ensures that
there is a unique best approximation pair $\left(a,b\right)\in A\times B$.
We note that Lemma \bref{lem:Lemma 13} below, which should be known
(and its proof is very simple and follows directly from the definition),
holds in any  normed space, and Proposition \bref{prop:Proposition 6}
below holds in any real inner product space. In what follows the underlying
space will be denoted by $X$. 
\begin{defn}
A subset $C$ of $X$ is called strictly convex if for all $x$ and
$y$ in $C$ and all $t\in\left(0,1\right)$, the point $tx+\left(1-t\right)y$
is in the interior of $C$. 
\end{defn}

\begin{lem}
\label{lem:Lemma 13} Given $m\in\mathbb{N}$ and $m$ strictly convex
subsets $C_{1},C_{2}\ldots,C_{m}$ of $X$, their intersection is
also strictly convex. 
\end{lem}

\begin{example}
Examples of bounded strictly convex sets: balls, ellipses, the intersection
of a paraboloid with a bounded strictly convex set, the intersection
of one arm of a hyperboloid with a bounded strictly convex set, level-sets
of coercive and continuous strictly convex functions (namely, sets
of the form $\left\{ x\in X\mid f\left(x\right)\leq\alpha\right\} $,
where $\alpha\in\mathbb{R}$ is fixed, $f:X\rightarrow\mathbb{R}$
is continuous and satisfies $\lim_{\left\Vert x\right\Vert \rightarrow\infty}f\left(x\right)=\infty$
(coercivity), and $f\left(tx+\left(1-t\right)y\right)<tf\left(x\right)+\left(1-t\right)f\left(y\right)$
for all $t\in\left(0,1\right)$ and $x,y\in X$ (strict convexity)). 
\end{example}

\begin{lem}
\label{lem:Lemma 5} Suppose that $A$ and $B$ are nonempty, closed and convex
subsets of a real inner product space $X$ and suppose that $a\in A$
and $b\in B$ satisfy $\left\Vert a-b\right\Vert =dist\left(A,B\right)>0$.
Then $A$ is contained in the closed half-space $\left\{ y\in X\mid\left\langle y-a,a-b\right\rangle \geq0\right\} $,
and $B$ is contained in the closed half-space $\left\{ y\in X\mid\left\langle y-b,b-a\right\rangle \geq0\right\} $.
Moreover, the point $a$ is a boundary point of $A$, and the point
$b$ is a boundary point of $B$. 
\end{lem}

\begin{proof}
We prove the assertion regarding $A$. The proof regarding $B$ is
similar. Since $\left\Vert a-b\right\Vert =dist\left(A,B\right)$,
we have $a=P_{A}\left(b\right)$. Thus, from the well known characterization
of the orthogonal projection \cite[Theorem 1.2.4]{Ceg-book}, any
$y\in A$ satisfies $\left\langle y-a,a-b\right\rangle \geq0$, namely
$A\subseteq\left\{ y\in X\mid\left\langle y-a,a-b\right\rangle \geq0\right\} $.

To see that $a$ is a boundary point of $A$, let $D$ be an arbitrary
(non-degenerate) ball centered at $a$, with say, radius $r>0$. Since
$a\in A\cap D$, it remains to show that $D$ contains points outside
$A$. Indeed, let $t\in\left(0,r/\left\Vert b-a\right\Vert \right)$.
Then $a+t\left(b-a\right)\in D$ but $\left\langle \left(a+t\left(b-a\right)\right)-a,a-b\right\rangle =-t\left\Vert a-b\right\Vert ^{2}<0$,
and so $a+t\left(b-a\right)$ is not in the set $\left\{ y\in X\mid\left\langle y-a,a-b\right\rangle \geq0\right\} $.
Since $A\subseteq\left\{ y\in X\mid\left\langle y-a,a-b\right\rangle \geq0\right\} $,
by the previous paragraph, we conclude that $a+t\left(b-a\right)$
cannot be in $A$, as required. Therefore, $a$ is a boundary point
of $A$. 
\end{proof}
\begin{prop}
\label{prop:Proposition 6} Let $X$ be a real inner product space and suppose that $A$ and $B$ are nonempty subsets of $X$.
(i) There is at least one best approximation pair $\left(a,b\right)\in A\times B$
if and only if $dist\left(A,B\right)$ is attained. (ii) If $A$ 
and $B$ are closed, strictly convex, and $dist\left(A,B\right)>0$,
then there is at most one best approximation pair $\left(a,b\right)\in A\times B$ relative to $(A,B)$. 
(iii) If $A$ and $B$ are closed, strictly convex, satisfy $dist\left(A,B\right)>0$, and the distance between
them is attained, then there is exactly one best approximation pair
relative to $(A,B)$. In particular, there is exactly
one best approximation pair relative to $(A,B)$ whenever
$A\subseteq X$ and $B\subseteq X$ are nonempty, strictly convex,
compact and $dist\left(A,B\right)>0$. 
\end{prop}

\begin{proof}
Part (i) is just a restatement of the assertion that $dist\left(A,B\right)$
is attained.

We now turn to Part (ii). Suppose that $\left(a,b\right)\in A\times B$
and $\left(\tilde{a},\tilde{b}\right)\in A\times B$ are two arbitrary
best approximation pairs. We will show that $\left(a,b\right)=\left(\tilde{a},\tilde{b}\right)$.
Since both $\left(a,b\right)$ and $\left(\tilde{a},\tilde{b}\right)$
are best approximation pairs relative to $(A,B)$, we have $\left\Vert a-b\right\Vert =dist\left(A,B\right)=\left\Vert \tilde{a}-\tilde{b}\right\Vert $.
Let $t\in\left(0,1\right)$ be arbitrary and let  $a_{t}:=(1-t)a+t\tilde{a}$
and  $b_{t}:=(1-t)b+t\tilde{b}$. By the convexity of $A$ and $B$
we have $a_{t}\in A$ and $b_{t}\in B$, and therefore, in particular,
$dist\left(A,B\right)\leq\left\Vert a_{t}-b_{t}\right\Vert $. On
the other hand, the triangle inequality and the assumption that $\left(a,b\right)$
and $\left(\tilde{a},\tilde{b}\right)$ are best approximation pairs relative to $(A,B)$ imply that  $\left\Vert a_{t}-b_{t}\right\Vert \leq(1-t)\left\Vert a-b\right\Vert +t\left\Vert \tilde{a}-\tilde{b}\right\Vert =(1-t)dist\left(A,B\right)+tdist\left(A,B\right)=dist\left(A,B\right)$.
Hence, $\left\Vert a_{t}-b_{t}\right\Vert =dist\left(A,B\right)$
and $\left(a_{t},b_{t}\right)$ is a best approximation pair too.
Thus, Lemma \bref{lem:Lemma 5} implies that $a_{t}$ is a boundary
point of $A$ and $b_{t}$ is a boundary point of $B$.

Suppose, by way of contradiction, that $a\neq\tilde{a}$. Then $a_{t}$
is strictly inside the line segment $\left[a,\tilde{a}\right]$, and
hence the strict convexity of $A$ implies that $a_{t}$ is in the
interior of $A$, contradicting the claim proved in the previous paragraphs
that $a_{t}$ is a boundary point of $A$. Consequently $a=\tilde{a}$,
and similarly $b=\tilde{b}$. Therefore, $\left(a,b\right)=\left(\tilde{a},\tilde{b}\right)$,
and hence all the best approximation pairs (if they exist) are equal.
Thus, there is at most one best approximation pair.

Now we prove Part (iii). From Part (ii)
we know that there is at most one best approximation pair, and from Part
(i) there is at least one best approximation pair. We conclude therefore
that there is exactly one best approximation pair $\left(a,b\right)\in A\times B$ relative to $(A,B)$. This conclusion holds when one considers the final case mentioned
in Part (iii) since, as is well-known and easily follows from compactness,
$dist\left(A,B\right)$ is attained whenever $A$ and $B$ are nonempty
and compact (one simply takes a cluster point of $\left\{ \left(a^{k},b^{k}\right)\right\} _{k=1}^{\infty}$,
where $\left(a^{k},b^{k}\right)\in A\times B$ satisfies $\left\Vert a^{k}-b^{k}\right\Vert <dist\left(A,B\right)+\left(1/k\right)$
for all $k\in\mathbb{N}$). 
\end{proof}
\begin{rem}
The strict convexity assumption in Proposition \bref{prop:Proposition 6}(iii)
is sufficient but not necessary, as shown in Figure \bref{fig:BestApproxPairNonStrictlyConvex}. 
\end{rem}

\begin{figure}[t]
\begin{minipage}[t]{1\textwidth}
\begin{center}{\includegraphics[trim=140 680 200 40, clip=true, scale=1]{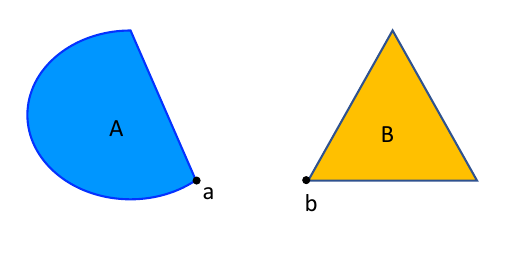}}
\end{center}
 \caption{A case where there is a unique best approximation pair $(a,b)$ despite the fact that $A$ and $B$ are not strictly convex}
\label{fig:BestApproxPairNonStrictlyConvex}
\end{minipage}
\end{figure}

\section{The Simultaneous-HLWB projection operator\label{sec:Features}}

Inspired by the iterative process of \beqref{eq:simult-HLWB} which
represents the simultaneous HLWB (S-HLWB) algorithm, we define the
Simultaneous-HLWB operator and introduce, in the sequel, some of its
properties. Given are two families of nonempty, closed and convex
sets $\mathcal{A}:=\left\{ A_{i}\right\} _{i=1}^{I}$ and $\mathcal{B}:=\left\{ B_{j}\right\} _{j=1}^{J}$
in a Euclidean space $\mathcal{H}$ for some positive integers $I$
and $J$. Assume that $A:=\bigcap_{i=1}^{I}A_{i}\neq\emptyset$ and
$B:=\bigcap_{j=1}^{J}B_{j}\neq\emptyset$ but $A\bigcap B=\emptyset$,
where $A_{i},B_{j}\subseteq B\left[0,\rho\right]$ for some $\rho>0$.
Our overall goal is to find a best approximation pair relative to
$A$ and $B$. 
\begin{defn}
\begin{flushleft}
We denote by 
\begin{equation}
M_{\mathcal{\mathcal{G}},\tau,w}[d]:=\tau d+\left(1-\tau\right)\sum_{\ell=1}^{L}w_{\ell}P_{C_{\ell}},\,L\in\mathbb{N}\label{eq:M_c}
\end{equation}
the operators $M_{\mathcal{\mathcal{G}},\tau}[d]:\mathcal{H\rightarrow H}$,
with a fixed anchor point $d\in\mathcal{H},$ with respect to the
family $\mathcal{G}:=\left\{ C_{\ell}\right\} _{\ell=1}^{L}$ of nonempty
closed convex sets such that $\bigcap_{\ell=1}^{L}C_{\ell}\neq\emptyset$,
such that $\tau\in(0,1)$  and such that  $w=(w_{\ell})_{\ell=1}^{L}$
is a positive weight vector. We call such operators \textbf{simultaneous-HLWB
operators}. I.e., for $x\in\mathcal{H},$ 
\begin{equation}
M_{\mathcal{\mathcal{G}},\tau,w}[d](x):=\tau d+\left(1-\tau\right)\sum_{\ell=1}^{L}w_{\ell}P_{C_{\ell}}(x).\label{eq:M_c(x)}
\end{equation}
\par\end{flushleft}
\end{defn}

\section{The alternating simultaneous HLWB algorithm\label{sec:The method} }

Motivated by the strategy of the original algorithm in \cite[Algorithm in Equation (2)]{bauschke1996approximation},
we present here our new method for finding a best approximation pair
in a finite-dimensional Hilbert space, namely the alternating simultaneous
HLWB algorithm (A-S-HLWB). But before doing so, at the end of this
section, we need a technical preparation.

First we look at products of S-HLWB operators. 
\begin{defn}
\label{def:generic-Qc} Given a natural number $L$,  a family  $\mathcal{\mathcal{G}}:=\left\{ C_{\ell}\right\} _{\ell=1}^{L}$
of nonempty, closed and convex sets such that  $C:=\bigcap_{\ell=1}^{L}C_{\ell}\neq\emptyset$,
a positive weight vector $w=(w_{\ell})_{\ell=1}^{L}$, a sequence  $\tau^{\mathcal{G}}:=\left\{ \tau_{t}^{\mathcal{\mathcal{G}}}\right\}_{t=0}^{\infty}$ of steering parameters, and some anchor
point $d\in\mathcal{H}$, define, for any $q=0,1,2,\ldots,$ the operators
$Q_{\mathcal{\mathcal{G}},\tau^{\mathcal{G}},q,w}[d]$, associated with $\mathcal{\mathcal{G}},$
by

\begin{align}
Q_{\mathcal{\mathcal{G}},\tau^{\mathcal{G}},q,w}[d]:= & \prod_{t=0}^{q}M_{\mathcal{\mathcal{G}},\tau_{t}^{\mathcal{\mathcal{G}}},w}[d]=M_{\mathcal{\mathcal{G}},\tau_{q}^{\mathcal{\mathcal{G}}},w}[d]\left(M_{\mathcal{\mathcal{G}},\tau_{q-1}^{\mathcal{G}},w}[d]\left(\cdots\left(M_{\mathcal{\mathcal{G}},\tau_{0}^{\mathcal{\mathcal{G}}},w}[d]\right)\right)\right).\label{eq:Q_C}
\end{align}
\end{defn}

When we choose in Definition \bref{def:generic-Qc} $\mathcal{\mathcal{G}}:=\mathcal{A}$, a steering parameter sequence $\tau^{\mathcal{A}}:=\left\{ \tau_{t}^{\mathcal{\mathcal{A}}}\right\}_{t=0}^{\infty}$,
$q:=s\in\mathbb{N}$, $L:=I,$ a positive weight vector $\alpha=(\alpha_{i})_{i=1}^{I}$,
and $d=u\in\mathcal{H}$, we obtain from the family $\mathcal{A}:=\left\{ A_{i}\right\} _{i=1}^{I}$
and its associated steering parameters sequence  $\left\{ \tau_{g}^{\mathcal{A}}\right\} _{g=0}^{\infty}$
the operators  $Q_{\mathcal{A},\tau^{\mathcal{A}},s,\alpha}[u]$ as follows,

\begin{equation}
Q_{\mathcal{A},\tau^{\mathcal{A}},s,w}[u]:=\prod_{t=0}^{s}M_{\mathcal{A},\tau_{t}^{\mathcal{A}},w}[u].\label{eq:QA}
\end{equation}

When we choose in the Definition \bref{def:generic-Qc} $\mathcal{\mathcal{G}}:=\mathcal{B}$, a steering parameter sequence $\tau^{\mathcal{B}}:=\left\{ \tau_{t}^{\mathcal{\mathcal{B}}}\right\}_{t=0}^{\infty}$, 
$q:=r\in\mathbb{N}$, $L:=J,$ a positive weight vector $\beta=(\beta_{j})_{j=1}^{J}$,
and $d=v\in\mathcal{H},$ we obtain from the family $\mathcal{B}:=\left\{ B_{i}\right\} _{i=1}^{I}$
and its associated steering parameters sequence  $\left\{ \tau_{h}^{\mathcal{B}}\right\} _{h=0}^{\infty}$
the operators  $Q_{\mathcal{B},\tau^{\mathcal{B}},r,\beta}[v]$ as follows, 

\begin{equation}
Q_{\mathcal{B},\tau^{\mathcal{B}},r,\beta}[v]:=\prod_{t=0}^{r}M_{\mathcal{B},\tau_{t}^{\mathcal{B}},\beta}[v].\label{eq:QB}
\end{equation}

Each of these operators is formed by a finite composition of operators\textcolor{black}{{}
and we will prove the existence of fixed points for them.}

In the following lemma we show the relation between the operator (\bref{eq:Q_C}) and the iterative process (\bref{eq:simult-HLWB}). 
\begin{lem}\label{lem:Xn=00003D00003DQn} Let $X$ be a nonempty, closed and convex subset of $\mathcal{H}$. Given a natural number $L$, suppose that $\mathcal{G}:=\{C_1,\ldots,C_L\}$ is a family of nonempty, closed and convex subsets of $X$. Let $T_{\ell}:=P_{\ell}:=P_{C_{\ell}}$
for all $\ell\in\left\{ 1,2,\ldots,L\right\}$. Given two points $d$ and $x^0$ in $X$, a positive weight vector $w$ and a sequence $\{\tau_k\}_{k=0}^{\infty}$ of steering parameters, consider the sequence $\left\{ x^{k}\right\}_{k=0}^{\infty}$ generated by (\bref{eq:simult-HLWB}), namely 
\begin{equation}
x^{k+1}=\tau_{k}d+\left(1-\tau_{k}\right)\sum_{\ell=1}^{L}w_{\ell}P_{\ell}x^{k}.\label{eq:X-sequence}
\end{equation}
Given $y^0\in X$ and another sequence $\tau^{\mathcal{G}}:=\{\tau_k^{\mathcal{G}}\}_{k=0}^{\infty}$ of steering parameters, consider the sequence $\{y^k\}_{k=0}^{\infty}$ defined using the operator (\bref{eq:Q_C}) of Definition \bref{def:generic-Qc}, namely 
\begin{equation}
y^{k+1}:=Q_{\mathcal{\mathcal{G}},\tau^{\mathcal{G}},k,w}[d]\left(y^{0}\right)=\prod_{t=0}^{k}M_{\mathcal{\mathcal{G}},\tau_{t}^{\mathcal{\mathcal{G}}},w}[d]\left(y^{0}\right).\label{eq:Y-sequence}
\end{equation}
If $y^{0}=x^{0}$ and\textup{ $\tau_{k}=\tau_{k}^{\mathcal{G}}$} for all nonnegative integers\textup{ $k$, then} $y^{k}=x^{k}$ for every $k\in\mathbb{N}\cup\{0\}$. In addition, if $C:=\cap_{\ell=1}^L C_{\ell}\neq\emptyset$, then 
\begin{equation}
\lim_{k\rightarrow\infty}x^{k}=\lim_{k\rightarrow\infty}y^{k}=P_{C}d.\label{eq:5.6a}
\end{equation}
\end{lem}

\begin{proof}
We first show, by mathematical induction, that $x^{k}=y^{k}$ for
all $k\in\mathbb{N}\cup\{0\}$. By (\bref{eq:X-sequence}) we have 
\begin{equation}
x^{1}=\tau_{0}d+\left(1-\tau_{0}\right)\sum_{\ell=1}^{L}w_{\ell}P_{\ell}(x^{0}),\label{eq:5.7a}
\end{equation}
and from (\bref{eq:Q_C}) and ({eq:X-sequence}) we have, for all $k\in\mathbb{N}$,
\begin{equation}
x^{k+1}=M_{\mathcal{\mathcal{G}},\tau_{k},w}[d]\left(x^{k}\right)=\tau_{k}d+\left(1-\tau_{k}\right)\sum_{\ell=1}^{L}w_{\ell}P_{\ell}\left(x^{k}\right).\label{eq:Xn}
\end{equation}
By (\bref{eq:Y-sequence}) we have 
\begin{equation}
y^{1}=M_{\mathcal{\mathcal{G}},\tau_{0}^{\mathcal{\mathcal{G}}},w}[d]\left(y^{0}\right)=\tau_{0}^{\mathcal{G}}d+\left(1-\tau_{0}^{\mathcal{G}}\right)\sum_{\ell=1}^{L}w_{\ell}P_{C_{\ell}}(y^{0}),\label{eq:5.9a}
\end{equation}
and from (\bref{eq:Y-sequence}) again we have, for all $k\in\mathbb{N}$, 
\begin{multline}
y^{k+1}=M_{\mathcal{\mathcal{G}},\tau_{k}^{\mathcal{G}},w}[d]\left(\prod_{t=0}^{k-1}M_{\mathcal{\mathcal{G}},\tau_{t}^{\mathcal{\mathcal{G}}},w}[d]\left(y^{0}\right)\right)=M_{\mathcal{\mathcal{G}},\tau_{k}^{\mathcal{G}},w}[d]\left(y^{k}\right)\\
=\tau_{k}^{\mathcal{G}}d+\left(1-\tau_{k}^{\mathcal{G}}\right)\sum_{\ell=1}^{L}w_{\ell}P_{\ell}\left(y^{k}\right).\label{eq:Yn}
\end{multline}
From (\bref{eq:5.7a}), (\bref{eq:5.9a}) and the assumptions that $x^{0}=y^{0}$
and $\tau_{0}=\tau_{0}^{\mathcal{G}}$, it follows that $x^{1}=y^{1}$.
Now we assume that the equality between $x^{\ell}$ and $y^{\ell}$
holds for all integers $\ell$ from $0$ to $k$. From (\bref{eq:Xn}),
(\bref{eq:Yn}) and the assumptions that $x^{k}=y^{k}$ and $\tau_{k}=\tau_{k}^{\mathcal{G}}$,
it follows that $x^{k+1}=y^{k+1}$ as well, as required.  

Finally, since, according to Proposition \bref{prop:P_C=00003D00003DNE}, we have $C_{\ell}=\textnormal{Fix}(T_{\ell})$ for all $\ell\in\{1,\ldots,L\}$, since we assume that $C:=\cap_{\ell=1}^L C_{\ell}\neq\emptyset$, and since  $x^{k}=y^{k}$ for all nonnegative integers $k$, we obtain (\bref{eq:5.6a}) from Theorem ~\bref{cor:simultaneous HLWB} (where there $x:=d$). 
\end{proof}

\noindent We shall employ the operators of (\bref{eq:M_c}) and (\bref{eq:Q_C}),
from here onward, always with an anchor point that is identical with
the initial point on which the operators act. To this end we define
$d:=x$ in (\bref{eq:M_c}) and obtain the new operator

\noindent 
\begin{equation}
\widehat{M}_{\mathcal{\mathcal{G}},\tau,w}(x):=M_{\mathcal{\mathcal{G}},\tau,w}[x](x)=\tau x+\left(1-\tau\right)\sum_{\ell=1}^{L}w_{\ell}P_{C_{\ell}}(x),\label{eq:M_C(d=00003D00003Dx)}
\end{equation}
thus,

\begin{equation}
\widehat{M}_{\mathcal{\mathcal{G}},\tau,w}=\tau\textup{Id}+\left(1-\tau\right)\sum_{\ell=1}^{L}w_{\ell}P_{C_{\ell}}.\label{eq:genericMhat}
\end{equation}

\noindent For the choice $d:=x$ and $\tau^{\mathcal{G}}:=\{\tau_k^{\mathcal{G}}\}_{k=0}^{\infty}$ 
in (\bref{eq:Q_C}) we obtain the new operator 
\begin{equation}
\begin{alignedat}[b]{1}\widehat{Q}_{\mathcal{\mathcal{G}},\tau^{\mathcal{G}},q,w}(x):=Q_{\mathcal{\mathcal{G}},\tau^{\mathcal{G}},q,w}[x](x)= & \prod_{t=0}^{q}M_{\mathcal{\mathcal{G}},\tau_{t}^{\mathcal{\mathcal{G}}},w}[x](x)=\prod_{t=0}^{q}\widehat{M}_{\mathcal{\mathcal{G}},\tau_{t}^{\mathcal{\mathcal{G}}},w}(x)\\
= & \prod_{t=0}^{q}\left(\tau_{t}^{\mathcal{\mathcal{G}}}\textup{Id}+\left(1-\tau_{t}^{\mathcal{\mathcal{G}}}\right)\sum_{\ell=1}^{L}w_{\ell}P_{C_{\ell}}\right)(x).
\end{alignedat}
\label{eq:Q_C(x)=00003D00003DQ_C(x,x)}
\end{equation}
This yields, for the choice $\mathcal{G}=\mathcal{A}$, $q=s\in\mathbb{N}$
and a positive weight vector $\alpha=(\alpha_{i})_{i=1}^{I}$, 

\noindent 
\begin{equation}
\widehat{Q}_{\mathcal{A},\tau^{\mathcal{A}},s,\alpha}(x):=Q_{\mathcal{A},\tau^{\mathcal{A}},s,\alpha}[x](x),\label{eq:Q_A(x)=00003D00003DQ_A(x,x)}
\end{equation}
and, for the choice $\mathcal{G}=\mathcal{\mathcal{B}}$,  $q=r\in\mathbb{N}$
and a positive weight vector $\beta=(\beta{}_{i})_{i=1}^{J}$, 

\noindent 
\begin{equation}
\widehat{Q}_{\mathcal{B},\tau^{\mathcal{B}},r,\beta}(x):=Q_{\mathcal{B},\tau^{\mathcal{B}},r,\beta}[x](x).\label{eq:Q_B(x)=00003D00003DQ_B(x,x)}
\end{equation}

\begin{lem}
\label{lem:M_C=00003D00003DNE}The operator of (\bref{eq:M_C(d=00003D00003Dx)})
and the operator of (\bref{eq:Q_C(x)=00003D00003DQ_C(x,x)}) are NE
operators. As a matter of fact, both of these operators are averaged.
\end{lem}

\begin{proof}
By Lemma \bref{lem:comp of contra is a contra}(i) a convex combination
of NE operators is NE and hence $\sum_{\ell=1}^{L}w_{\ell}P_{C_{\ell}}$
is NE. Since $\tau\in(0,1)$ we see from (\bref{eq:genericMhat}) that
$\widehat{M}_{\mathcal{\mathcal{G}},\tau,w}$ is a convex combination
of the identity operator and a NE operator, namely $\widehat{M}_{\mathcal{\mathcal{G}},\tau,w}$
is averaged. Thus because of Lemma \bref{lem:comp of contra is a contra}(i)
and because the identity operator is NE, also $\widehat{M}_{\mathcal{\mathcal{G}},\tau,w}$
is NE. By Lemma \bref{lem:comp of contra is a contra}(ii) the operator
$\widehat{Q}_{\mathcal{\mathcal{G}},\tau^{\mathcal{G}},q,w}$ is also NE. In fact, $\widehat{Q}_{\mathcal{\mathcal{G}},\tau^{\mathcal{G}},q,w}$
is averaged according to \cite[Proposition 4.44]{BC17} since it
is a finite composition of averaged operators (namely the operators
$\widehat{M}_{\mathcal{\mathcal{G}},\tau_{t}^{\mathcal{\mathcal{G}}},w},\,t\in\left\{ 0,\ldots,q\right\} $). 
\end{proof}
A key point in our development is the property described in the next
lemma about the successive application of these two operators. 
\begin{lem}
\label{lem:QAQB=00003D00003DPAPB} For every $x\in\mathcal{H}$ 

\begin{equation}
\lim_{r\rightarrow\infty}\left(\lim_{s\rightarrow\infty}\left(\widehat{Q}_{\mathcal{B},\tau^{\mathcal{B}},r,\beta}\left(\widehat{Q}_{\mathcal{A},\tau^{\mathcal{A}},s,\alpha}\left(x\right)\right)\right)\right)=P_{B}\left(P_{A}\left(x\right)\right).
\end{equation}
\end{lem}

\begin{proof}
Given $x$ in $\mathcal{H}$, Lemma \bref{lem:Xn=00003D00003DQn} with
$y^{0}:=x^{0}:=d:=x$ and $y^{s+1}:=\widehat{Q}_{\mathcal{A},\tau^{\mathcal{A}},s,\alpha}(x)=Q_{\mathcal{A},\tau^{\mathcal{A}},s,\alpha}[d](x)$
implies that  

\begin{equation}
\lim_{s\rightarrow\infty}\left(\widehat{Q}_{\mathcal{A},\tau^{\mathcal{A}},s,\alpha}\left(x\right)\right)=P_{A}\left(x\right).\label{eq:5.24}
\end{equation}
 Similarly, given $y$ in $\mathcal{H}$, Lemma \bref{lem:Xn=00003D00003DQn}
with $y^{0}:=x^{0}:=d:=y$ and $y^{r+1}:=\widehat{Q}_{\mathcal{B},\tau^{\mathcal{B}},r,\beta}(y)=Q_{\mathcal{B},\tau^{\mathcal{B}},r,\beta}[d](y)$
implies that 

\begin{equation}
\lim_{r\rightarrow\infty}\left(\widehat{Q}_{\mathcal{B},\tau^{\mathcal{B}},r,\beta}\left(y\right)\right)=P_{B}\left(y\right).\label{eq:5.25}
\end{equation}
Thus, due to the continuity of  $\widehat{Q}_{\mathcal{B},\tau^{\mathcal{B}},r,\beta}$
for all $r\in\mathbb{N}$, which follows from the nonexpansivity property
of the operators (Lemma \bref{lem:M_C=00003D00003DNE}),

\begin{align}
\lim_{r\rightarrow\infty}\left(\lim_{s\rightarrow\infty}\left(\widehat{Q}_{\mathcal{B},\tau^{\mathcal{B}},r,\beta}\left(\widehat{Q}_{\mathcal{A},\tau^{\mathcal{A}},s,\alpha}\left(x\right)\right)\right)\right) & =\lim_{r\rightarrow\infty}\left(\widehat{Q}_{\mathcal{B},\tau^{\mathcal{B}},r,\beta}\left(\lim_{s\rightarrow\infty}\left(\widehat{Q}_{\mathcal{A},\tau^{\mathcal{A}},s,\alpha}\left(x\right)\right)\right)\right)\nonumber \\
 & =\lim_{r\rightarrow\infty}\left(\widehat{Q}_{\mathcal{B},\tau^{\mathcal{B}},r,\beta}\left(P_{A}\left(x\right)\right)\right)=P_{B}\left(P_{A}\left(x\right)\right).
\end{align}
\end{proof}
Next is the classical theorem named after Dini, see, e.g., \cite[Theorem 7.13, page 150]{rudin1976principles}
or \cite[Theorem 8.2.6]{bartle2011introduction}, which tells us
when pointwise convergence of a sequence of functionals (i.e., operators
into the real line) implies its uniform convergence. 
\begin{thm}
\label{thm:Dini=00003D002019s-classical-Theorem}\textbf{ Dini's Theorem}.
Let $K$ be a compact metric space. Let $f:K\rightarrow\mathbb{R}$
be a continuous functional and let $\left\{ f_{k}\right\} _{k=1}^{\infty}$,
$f_{k}:K\rightarrow\mathbb{R}$, be a sequence of continuous functionals.
If $\left\{ f_{k}\right\} _{k=1}^{\infty}$ is monotonic and converges
pointwise to $f$, then $\left\{ f_{k}\right\} _{k=1}^{\infty}$ converges
uniformly to $f$. 
\end{thm}

We will need the following generalization of Dini's Theorem for general
operators (not necessarily functionals), which we derive from Theorem
\bref{thm:Dini=00003D002019s-classical-Theorem}. This proposition
might be known, but we have not seen it in the literature. 
\begin{prop}
\label{prop:generalized-Dini's-Theorem} Let $(K,d_{K})$ be a compact
metric space and let $\left(Y,d_{Y}\right)$ be a metric space. For
each $k\in\mathbb{N}$ suppose that $T_{k}:K\to Y$ is a continuous
operator and assume that there is a continuous operator $T:K\to Y$
such that $\left\{ T_{k}\right\} _{k=1}^{\infty}$ converges pointwise
to $T$. If $d_{Y}(T_{k+1}(x),T(x))\leq d_{Y}(T_{k}(x),T(x))$ for
all $k\in\mathbb{N}$ and all $x\in K$, then $\left\{ T_{k}\right\} _{k=1}^{\infty}$
converges uniformly to $T$. 
\end{prop}

\begin{proof}
Denote $f_{k}(x):=d_{Y}(T_{k}(x),T(x))$ for all $k\in\mathbb{N}$
and $x\in K$. By the continuity of $T_{k}$ and $T$ and the continuity
of the metric, it follows that $f_{k}:K\to[0,\infty)$ is continuous
for every $k\in\mathbb{N}$. Since $\left\{ T_{k}\right\} _{k=1}^{\infty}$
converges pointwise to $T$, it follows that $0=\lim_{k\rightarrow\infty}d_{Y}(T_{k}(x),T(x))=\lim_{k\rightarrow\infty}f_{k}(x)=\lim_{k\rightarrow\infty}\left|f_{k}\left(x\right)-0\right|$,
namely $\left\{ f_{k}\right\} _{k=1}^{\infty}$ converges pointwise
to $f\equiv0$. Since we assume that $d_{Y}(T_{k+1}(x),T(x))\leq d_{Y}(T_{k}(x),T(x))$
for all $x\in K$ and $k\in\mathbb{N},$ we have $f_{k+1}(x)=d_{Y}(T_{k+1}(x),T(x))\leq d_{Y}(T_{k}(x),T(x))=f_{k}(x)$
for all $x\in K$ and $k\in\mathbb{N}$, namely the sequence $\left\{ f_{k}\right\} _{k=1}^{\infty}$
is monotone. Hence, using the fact that $K$ is compact, we conclude
from Theorem \bref{thm:Dini=00003D002019s-classical-Theorem} that
$\left\{ f_{k}\right\} _{k=1}^{\infty}$ converges uniformly to the
zero function. Thus, given $\varepsilon>0$, there exists some $k_{\varepsilon}\in\mathbb{N}$
such that $d_{Y}(T_{k}(x),T(x))=\left|f_{k}\left(x\right)-0\right|<\varepsilon$
for all $x\in K$ and all $k_{\varepsilon}\leq k\in\mathbb{N}$, namely
$\left\{ T_{k}\right\} _{k=1}^{\infty}$ does converge uniformly to
$T$. 
\end{proof}
\begin{lem}
\label{lem:Fix(M_c)} The fixed point set of the operator \textup{$\widehat{M}_{\mathcal{\mathcal{G}},\tau,w}$}
of (\bref{eq:M_C(d=00003D00003Dx)}) is equal to the, assumed nonempty,
intersection $\bigcap_{\ell=1}^{L}C_{\ell}$. 
\end{lem}

\begin{proof}
Let us denote $S\left(x\right):=\sum_{\ell=1}^{L}w_{\ell}P_{C_{\ell}}\left(x\right)$.
The identity operator is NE and any $x\in\mathcal{H}$ is a fixed
point of it. By Proposition \bref{prop:P_C=00003D00003DNE} $\textup{Fix}P_{C_{\ell}}=C_{\ell}$,
hence $\bigcap_{\ell=1}^{L}\textup{Fix}P_{C_{\ell}}=\bigcap_{\ell=1}^{L}C_{\ell}\neq\emptyset$.
The projections $P_{C_{\ell}}$ are NEs by Proposition \bref{prop:P_C=00003D00003DNE}.
By Proposition \bref{prop:Fix(combi)=00003D00003Dinter(Fix)} we have
$\textup{Fix}S=\bigcap_{\ell=1}^{L}\textup{Fix}P_{C_{\ell}}$, therefore,
$\textup{Fix}\textup{Id}\bigcap\textup{Fix}S=\bigcap_{\ell=1}^{L}C_{\ell}\neq\emptyset$.
The operator $S$ is a convex combination of the NE operators $P_{C_{\ell}}$,
so, it is also NE. For each $\tau\in(0,1)$ the operator $\widehat{M}_{\mathcal{\mathcal{G}},\tau,w}$
is an averaged NE operator since 
\begin{equation}
\widehat{M}_{\mathcal{\mathcal{G}},\tau,w}=\tau\textup{Id}+\left(1-\tau\right)S.\label{eq:Comb(R+S)}
\end{equation}
Hence, by Proposition \bref{prop:Fix(combi)=00003D00003Dinter(Fix)},
we have 
\begin{equation}
\textup{Fix}\widehat{M}_{\mathcal{\mathcal{G}},\tau,w}=\textup{Fix}\textup{Id}\bigcap\textup{Fix}S=\bigcap_{\ell=1}^{L}C_{\ell}.\label{eq:FixM=00003D00003DINT-C}
\end{equation}
\end{proof}
In the following lemma we prove that the fixed point set of a composition
of operators of the kind defined by (\bref{eq:Comb(R+S)}) is nonempty. 
\begin{lem}
\label{lem:Q_A=00003D00003DsQNE}For each $q\in\mathbb{N},$ we have
\begin{equation}
\textup{Fix}\widehat{Q}_{\mathcal{\mathcal{G}},\tau^{\mathcal{G}},q,w}=\bigcap_{\ell=1}^{L}C_{\ell}.
\end{equation}
\end{lem}

\begin{proof}
Lemma \bref{lem:M_C=00003D00003DNE} implies that the operators $\widehat{M}_{\mathcal{\mathcal{G}},\tau_{t}^{\mathcal{\mathcal{G}}},w}$
are NEs for each $\tau_{t}^{\mathcal{\mathcal{G}}}\in\left(0,1\right)$.
By Lemma \bref{lem:Fix(M_c)} 
\begin{equation}
\textup{Fix}\widehat{M}_{\mathcal{\mathcal{G}},\tau_{t}^{\mathcal{\mathcal{G}}},w}=\bigcap_{\ell=1}^{L}C_{\ell}\neq\emptyset,\quad\textup{for each }\tau_{t}^{\mathcal{\mathcal{G}}}\in\left(0,1\right).\label{eq:FixMnonempty}
\end{equation}
By (\bref{eq:FixM=00003D00003DINT-C}) it is clear that 
\begin{equation}
\bigcap_{t=0}^{q}\mathit{\mathrm{Fix}}\widehat{M}_{\mathcal{\mathcal{G}},\tau_{t}^{\mathcal{\mathcal{G}}},w}=\bigcap_{t=0}^{q}\bigcap_{\ell=1}^{L}C_{\ell}=\bigcap_{\ell=1}^{L}C_{\ell}\neq\emptyset.\label{eq:IntersectFixM}
\end{equation}
We observe that (\bref{eq:Comb(R+S)}) implies that $\widehat{M}_{\mathcal{\mathcal{G}},\tau_{t}^{\mathcal{\mathcal{G}}},w}$
is an averaged NE operator for each $t\in\left\{ 1,2,\ldots,q\right\} $.
Since $\textup{Fix}\widehat{M}_{\mathcal{\mathcal{G}},\tau_{t}^{\mathcal{\mathcal{G}}},w}\neq\emptyset$,
as follows from (\bref{eq:FixMnonempty}), and since $\widehat{Q}_{\mathcal{\mathcal{G}},q,w}$
is a finite composition of the averaged NE operators $\widehat{M}_{\mathcal{\mathcal{G}},\tau_{t}^{\mathcal{\mathcal{G}}},w},$
$t\in\left\{ 1,2,\ldots,q\right\} $, it follows from (\bref{eq:IntersectFixM})
and Proposition \bref{prop:Fix(compo)=00003D00003Dinter(Fix)} that
\begin{equation}
\textup{Fix}\widehat{Q}_{\mathcal{\mathcal{G}},\tau^{\mathcal{G}},q,w}=\textup{Fix}\prod_{t=0}^{q}\widehat{M}_{\mathcal{\mathcal{G}},\tau_{t}^{\mathcal{\mathcal{G}}},w}=\bigcap_{t=0}^{q}\mathit{\mathrm{Fix}}\widehat{M}_{\mathcal{\mathcal{G}},\tau_{t}^{\mathcal{\mathcal{G}}},w}=\bigcap_{\ell=1}^{L}C_{\ell}.
\end{equation}
\end{proof}
\begin{lem}
\label{lem:UC-to-Pc} Under the conditions of Definition \bref{def:generic-Qc}, for each $k\in\mathbb{N}$ let $T_{k}:=\widehat{Q}_{\mathcal{G},\tau^{\mathcal{G}},k-1,w}$. Then $\left\{ T_{k}\right\} _{k=1}^{\infty}$
converges pointwise to $P_{C}$ and the convergence is uniform on
every compact set (and, in particular, on every closed ball). 
\end{lem}

\begin{proof}
Denote $T:=P_{C}$. Fix an arbitrary $x\in\mathcal{H}$ and let $z:=T(x)$.
Then $z=P_{C}(x)\in C=\cap_{\ell=1}^L C_{\ell}$, and so $\widehat{M}_{\mathcal{\mathcal{G}},\tau_{k}^{\mathcal{\mathcal{G}}},w}\left(z\right)=z$
for all $k\in\mathbb{N}$ by Lemma \bref{lem:Fix(M_c)}. Observe that,
by the definition of $T_{k+1}$, we have $T_{k+1}=\widehat{M}_{\mathcal{\mathcal{G}},\tau_{k}^{\mathcal{\mathcal{G}}},w}T_{k}$
whenever $k\geq1.$  Since $\widehat{M}_{\mathcal{\mathcal{G}},\tau_{k}^{\mathcal{\mathcal{G}}},w}$
is NE, according to Lemma \bref{lem:M_C=00003D00003DNE}, it follows
from the previous lines that for all $k\geq1$,  
\begin{align}
 & \left\Vert T_{k+1}(x)-T(x)\right\Vert =\left\Vert T_{k+1}(x)-z\right\Vert \nonumber \\
 & =\left\Vert \widehat{M}_{\mathcal{\mathcal{G}},\tau_{k}^{\mathcal{\mathcal{G}}},w}\left(T_{k}\left(x\right)\right)-\widehat{M}_{\mathcal{\mathcal{G}},\tau_{k}^{\mathcal{\mathcal{G}}},w}\left(z\right)\right\Vert \nonumber \\
 & \leq\left\Vert T_{k}(x)-z\right\Vert =\left\Vert T_{k}(x)-T(x)\right\Vert .\label{eq:5.42}
\end{align}

Now, if we choose $d:=y^{0}\in\mathcal{H}$ in (\bref{eq:Y-sequence}),
then we obtain the iterative process $y^{k+1}:=\widehat{Q}_{\mathcal{\mathcal{G}},k,w}\left(y^{0}\right)=T_{k+1}\left(y^{0}\right)$,
and Lemma \bref{lem:Xn=00003D00003DQn} then guarantees that $\lim_{k\rightarrow\infty}y^{k}=P_{C}\left(y^{0}\right)$
for every $y^{0}\in\mathcal{H}$. Thus, $\left\{ T_{k}\right\} _{k=1}^{\infty}$
converges pointwise to $P_{C}$. As a result of (\bref{eq:5.42}) and
the fact that $T_{k}$ is continuous for all $k$, as a composition
of continuous operators, Proposition \bref{prop:generalized-Dini's-Theorem}
yields the uniform convergence of $\left\{ T_{k}\right\} _{k=1}^{\infty}$
to $T=P_{C}$ on every compact set, and in particular on every closed
ball. 
\end{proof}
\begin{thm}
\label{thm:UC-to-P(B)P(A)} Let $\mathcal{H}$ be a finite-dimensional
real Hilbert space. \textcolor{black}{Let $I$ and $J$ be natural
numbers, $\tau^{\mathcal{A}}=\left\{ \tau_{k}^{\mathcal{A}}\right\} _{k=0}^{\infty}$
and $\tau^{\mathcal{B}}=\left\{ \tau_{k}^{\mathcal{B}}\right\} _{k=0}^{\infty}$ be two
sequences of steering parameters, and $\alpha=(\alpha_{i})_{i=1}^{I}$
and $\beta=(\beta_{j})_{j=1}^{J}$ be two positive weight vectors.} Suppose
that  $R_{k}:\mathcal{H}\rightarrow\mathcal{H}$ and $S_{k}:\mathcal{H}\rightarrow\mathcal{H}$, $k\in\mathbb{N}$ are
defined  for all $k\in\mathbb{N}\cup\left\{ 0\right\} $ by 

\begin{equation}
R_{k+1}:=\widehat{Q}_{\mathcal{A},\tau^{\mathcal{A}},k,\alpha}=\prod_{t=0}^{k}\widehat{M}_{\mathcal{A},\tau_{t}^{\mathcal{A}},\alpha}=\prod_{t=0}^{k}\left(\tau_{t}^{\mathcal{A}}\textup{Id}+\left(1-\tau_{t}^{\mathcal{A}}\right)\sum_{i=1}^{I}\alpha_{i}P_{A_{i}}\right)
\end{equation}
and 

\begin{equation}
S_{k+1}:=\widehat{Q}_{\mathcal{B},\tau^{\mathcal{B}},k,\beta}=\prod_{t=0}^{k}\widehat{M}_{\mathcal{B},\tau_{t}^{\mathcal{B}},\beta}=\prod_{t=0}^{k}\left(\tau_{t}^{\mathcal{B}}\textup{Id}+\left(1-\tau_{t}^{\mathcal{B}}\right)\sum_{j=1}^{J}\beta_{j}P_{B_{j}}\right).
\end{equation}
Further assume that there is some $\rho>0$ such that $A_{i},B_{j}\subseteq B\left[0,\rho\right]$
for all $i\in\left\{ 1,2,\ldots,I\right\} $ and $j\in\left\{ 1,2,\ldots,J\right\} $.
Then the sequence defined for all $k\in\mathbb{N}\cup\left\{ 0\right\} $
by 

\begin{align}
T_{k+1}:= & S_{k+1}R_{k+1}=\widehat{Q}_{\mathcal{B},\tau^{\mathcal{B}},k,\beta}\widehat{Q}_{\mathcal{A},\tau^{\mathcal{A}},k,\alpha}=\left(\prod_{t=0}^{k}\widehat{M}_{\mathcal{B},\tau_{t}^{\mathcal{B}},\beta}\right)\left(\prod_{t=0}^{k}\widehat{M}_{\mathcal{A},\tau_{t}^{\mathcal{A}},\alpha}\right)
\end{align}
converges uniformly to $P_{B}P_{A}$ on $B\left[0,\rho\right]$, and the sequence $\{R_{k+1}S_{k+1}\}_{k=0}^{\infty}$ converges uniformly to $P_AP_B$ on  $B[0,\rho]$. 
\end{thm}

\begin{proof}
By Lemma \bref{lem:UC-to-Pc}, with $\mathcal{G}=\mathcal{A}$ and
$C=A$, $\left\{ R_{k}\right\} _{k=1}^{\infty}$ converges uniformly
to $P_{A}$ on $B\left[0,\rho\right]$, i.e., for every $\varepsilon_{1}>0$,
there exists an integer $N_{1}$ such that  
\begin{equation}
\left\Vert R_{k+1}(x)-P_{A}(x)\right\Vert =\left\Vert \widehat{Q}_{\mathcal{A},\tau^{\mathcal{A}},k,\alpha}(x)-P_{A}(x)\right\Vert <\varepsilon_{1},\label{eq:epsilon-1}
\end{equation}

for all $k>N_{1}$ and all $x\in B\left[0,\rho\right]$.

By Lemma \bref{lem:UC-to-Pc} again, now with $\mathcal{G}=\mathcal{B}$
and $C=B$, $\left\{ S_{k}\right\} _{k=1}^{\infty}$ converges uniformly
to $P_{B}$ on $B\left[0,\rho\right]$, i.e., for every $\varepsilon_{2}>0$,
there exists an integer $N_{2}$ such that

\begin{equation}
\left\Vert S_{k+1}(y)-P_{B}(y)\right\Vert =\left\Vert \widehat{Q}_{\mathcal{B},\tau^{\mathcal{B}},k,\xi}(y)-P_{B}(y)\right\Vert <\varepsilon_{2},\label{eq:epsilon-2}
\end{equation}
for all $k>N_{2}$ and all $y\in B\left[0,\rho\right]$.

Now we prove that $y^{k}:=\widehat{Q}_{\mathcal{A},\tau^{\mathcal{A}},k,\alpha}\left(x\right)$
is in the ball $B\left[0,\rho\right]$ for all $k\in\mathbb{N}\cup\left\{ 0\right\} $
and for all $x\in B\left[0,\rho\right]$. Since $x\in B\left[0,\rho\right]$
and $A_{i}$ and $B_{j}$ are contained in $B\left[0,\rho\right]$
for all  $i\in\{1,2,\ldots,I\}$ and $j\in\{1,2,\ldots,J\}$, by
assumption, we have  $P_{A_{i}}\left(x\right)\in A_{i}\subseteq B\left[0,\rho\right]$
for all  $i\in\left\{ 1,2,\ldots,I\right\} $ and so, by the convexity
of $B\left[0,\rho\right]$ and the facts that $\alpha_{i}\in(0,1)$
for all $i\in\{1,2,\ldots,I\}$ and $\sum_{i=1}^{I}\alpha_{i}=1$, also $\sum_{i=1}^{I}\alpha_{i}P_{A_{i}}\left(x\right)\in B\left[0,\rho\right]$.
Again by the convexity of $B\left[0,\rho\right]$ and the fact that
$\tau_{0}^{\mathcal{\mathcal{A}}}\in\left(0,1\right)$, also  $\widehat{M}_{\mathcal{A},\tau_{0}^{\mathcal{A}},\alpha}(x)=\tau_{0}^{\mathcal{\mathcal{A}}}x+\left(1-\tau_{0}^{\mathcal{\mathcal{A}}}\right)\sum_{i=1}^{I}\alpha_{i}P_{A_{i}}\left(x\right)$ is
in $B\left[0,\rho\right]$, namely  $y^{0}=\widehat{M}_{\mathcal{A},\tau_{0}^{\mathcal{A}},\alpha}(x)\in B\left[0,\rho\right]$.
By induction on $k$, by the equality  $y^{k}=\widehat{Q}_{\mathcal{A},\tau^{\mathcal{A}},k,\alpha}(x)=\widehat{M}_{\mathcal{A},\tau^{\mathcal{A}},\alpha}(\widehat{Q}_{\mathcal{A},\tau^{\mathcal{A}},k-1,\alpha}(x))=\widehat{M}_{\mathcal{A},\tau^{\mathcal{A}},\alpha}(y^{k-1})$ and by arguments similar to the ones used in the case $k=0$, we see that indeed $y^{k}\in B\left[0,\rho\right]$ for every $k\in\mathbb{N}\cup\left\{ 0\right\} $.

Now we combine the results established above, where in the following
calculations $\varepsilon_{1}:=\varepsilon_{2}:=\dfrac{\varepsilon}{2}$
for some arbitrary $\varepsilon>0$, $k\in\mathbb{N}$ is larger than $\max\left\{ N_{1},N_{2}\right\} $,
the triangle inequality implies (\bref{eq:tri-ine}), the fact that
$P_{B}$ is NE by Proposition \bref{prop:P_C=00003D00003DNE} implies
(\bref{eq:P_B-NE}),  inequality (\bref{eq:Uni-Con}) follows from (\bref{eq:epsilon-1}) and (\bref{eq:epsilon-2}),
and  $y:=\widehat{Q}_{\mathcal{A},k,\tau^{\mathcal{A}},\alpha}\left(x\right)$
in (\bref{eq:epsilon-2}) (where $x\in B\left[0,\rho\right]$ is arbitrary
and we also use the fact that $\widehat{Q}_{\mathcal{A},\tau^{\mathcal{A}},k,\alpha}\left(x\right)\in B\left[0,\rho\right]$
for all $x\in B\left[0,\rho\right]$ as we showed earlier): 


\footnotesize

\begin{align}
 & \left\Vert \widehat{Q}_{\mathcal{B},\tau^{\mathcal{B}},k,\beta}\widehat{Q}_{\mathcal{A},\tau^{\mathcal{A}},k,\alpha}\left(x\right)-P_{B}P_{A}\left(x\right)\right\Vert \\
 &  =\left\Vert \widehat{Q}_{\mathcal{B},\tau^{\mathcal{B}},k,\beta}\left(\widehat{Q}_{\mathcal{A},\tau^{\mathcal{A}},k,\alpha}\left(x\right)\right)-P_{B}\left(P_{A}\left(x\right)\right)\right\Vert \\
 &  =\left\Vert \widehat{Q}_{\mathcal{B},\tau^{\mathcal{B}},k,\beta}\left(\widehat{Q}_{\mathcal{A},\tau^{\mathcal{A}},k,\alpha}\left(x\right)\right)-P_{B}\left(\widehat{Q}_{\mathcal{A},\tau^{\mathcal{A}},k,\alpha}\left(x\right)\right)+P_{B}\left(\widehat{Q}_{\mathcal{A},\tau^{\mathcal{A}},k,\alpha}\left(x\right)\right)-P_{B}\left(P_{A}\left(x\right)\right)\right\Vert \\
 & \leq\left\Vert \widehat{Q}_{\mathcal{B},\tau^{\mathcal{B}},k,\beta}\left(\widehat{Q}_{\mathcal{A},\tau^{\mathcal{A}},k,\alpha}\left(x\right)\right)-P_{B}\left(\widehat{Q}_{\mathcal{A},\tau^{\mathcal{A}},k,\alpha}\left(x\right)\right)\right\Vert +\left\Vert P_{B}\left(\widehat{Q}_{\mathcal{A},\tau^{\mathcal{A}},k,\alpha}\left(x\right)\right)-P_{B}\left(P_{A}\left(x\right)\right)\right\Vert \label{eq:tri-ine}\\ &  \leq\left\Vert \widehat{Q}_{\mathcal{B},\tau^{\mathcal{B}},k,\beta}\left(\widehat{Q}_{\mathcal{A},\tau^{\mathcal{A}},k,\alpha}\left(x\right)\right)-P_{B}\left(\widehat{Q}_{\mathcal{A},\tau^{\mathcal{A}},k,\alpha}\left(x\right)\right)\right\Vert +\left\Vert \widehat{Q}_{\mathcal{A},\tau^{\mathcal{A}},k,\alpha}\left(x\right)-P_{A}\left(x\right)\right\Vert \label{eq:P_B-NE}\\
 & \leq \varepsilon_1+\varepsilon_2=\varepsilon.\label{eq:Uni-Con}
\end{align}

\normalsize

Therefore, by Definition \bref{def:uniformly convergence}(ii), the
sequence $\left\{ T_{k+1}\right\} _{k=0}^{\infty}$ converges uniformly
to $P_{B}P_{A}$ on $B[0,\rho]$, i.e., 

\begin{equation}
\lim_{k\rightarrow\infty}T_{k+1}=\lim_{k\rightarrow\infty}\left(\widehat{Q}_{\mathcal{B},\tau^{\mathcal{B}},k,\beta}\widehat{Q}_{\mathcal{A},\tau^{\mathcal{A}},k,\alpha}\right)=P_{B}P_{A}.
\end{equation}
Similarly, the sequence $\{R_{k+1}S_{k+1}\}_{k=0}^{\infty}$ converges uniformly to $P_AP_B$ on $B[0,\rho]$. 
\end{proof}

Since an important component in our algorithm is that the number $q$
of successive projections onto each of the two intersections increases
from one sweep to the next, we prove the following theorem, which
is a cornerstone of our analysis that will follow, for $q\rightarrow\infty$. 
\begin{thm}
\textcolor{black}{\label{thm:adapted cheney and goldstein} Given
are two families of closed }and\textcolor{black}{{} convex sets}\textbf{\textcolor{black}{{}
}}\textcolor{black}{$\left\{ A_{i}\right\} _{i=1}^{I}$ and $\left\{ B_{j}\right\} _{j=1}^{J}$
in }a finite-dimensional real Hilbert space\textcolor{black}{{} $\mathcal{H},$
for some positive integers $I$ and $J.$ A}ssume that $A_{i},B_{j}\subseteq B\left[0,\rho\right]$
for all $i\subseteq\left\{ 1,2,\ldots,I\right\} $ and $j\subseteq\left\{ 1,2,\ldots,J\right\} $
for some $\rho>0$.\textcolor{black}{{} Given are also two sequences
of parameters $\left\{ \tau_{k}^{\mathcal{A}}\right\} _{k=0}^{\infty}$
and $\left\{ \tau_{k}^{\mathcal{B}}\right\} _{k=0}^{\infty}$ and
two positive weight vectors $\alpha=(\alpha_{i})_{i=1}^{I}$ and $\beta=(\beta_{j})_{j=1}^{J}$.
Assume that $A:=\cap_{i=1}^{I}A_{i}\neq\emptyset$ and $B:=\cap_{j=1}^{J}B_{j}\neq\emptyset$
but $A\cap B=\emptyset$. Let $\{\widehat{Q}_{\mathcal{A},\tau^{\mathcal{A}},k,\alpha}\}_{k=0}^{\infty}$
and  $\{\widehat{Q}_{\mathcal{B},\tau^{\mathcal{B}},k,\beta}\}_{k=0}^{\infty}$ be sequences
of operators defined in (\bref{eq:Q_A(x)=00003D00003DQ_A(x,x)}) and
in (\bref{eq:Q_B(x)=00003D00003DQ_B(x,x)}), respectively, and create
the sequences $\{\widehat{Q}_{\mathcal{B},\tau^{\mathcal{B}},k,\beta}\widehat{Q}_{\mathcal{A},\tau^{\mathcal{A}},k,\alpha}\}_{k=0}^{\infty}$
and $\{\widehat{Q}_{\mathcal{A},\tau^{\mathcal{A}},k,\alpha}\widehat{Q}_{\mathcal{B},\tau^{\mathcal{B}},k,\beta}\}_{k=0}^{\infty}$
of their products. } Then the limit operators of the sequences of products
exist, and the fixed point sets are:

 \textcolor{black}{
\begin{equation}
\textup{Fix}\left(\lim_{k\rightarrow\infty}\left(\widehat{Q}_{\mathcal{B},\tau^{\mathcal{B}},k,\beta}\widehat{Q}_{\mathcal{A},\tau^{\mathcal{A}},k,\alpha}\right)\right)=\left\{ z\in B\mid\left\Vert z-P_{A}\left(z\right)\right\Vert =dist\left(A,B\right)\right\} ,\label{eq:distance}
\end{equation}
and 
\begin{equation}
\textup{Fix}\left(\lim_{k\rightarrow\infty}\left(\widehat{Q}_{\mathcal{A},\tau^{\mathcal{A}},k,\alpha}\widehat{Q}_{\mathcal{B},\tau^{\mathcal{B}},k,\beta}\right)\right)=\left\{ z\in A\mid\left\Vert z-P_{B}\left(z\right)\right\Vert =dist\left(A,B\right)\right\} .\label{eq:distance 2}
\end{equation}
}
\end{thm}

\begin{proof}
Both $A$ and $B$ are nonempty (by assumption). Moreover, both of
them are compact since each of them is an intersection of closed and
bounded (hence compact) subsets. As a result, and as is well-known,
the distance $dist\left(A,B\right)$ between them is attained. Hence
Theorem \bref{thm:cheney and goldstein Thm 2} ensures that 
\begin{equation}
\textup{Fix\ensuremath{\left(P_{B}P_{A}\right)}}=\left\{ z\in B\mid\left\Vert z-P_{A}\left(z\right)\right\Vert =dist\left(A,B\right)\right\} 
\end{equation}
and 
\begin{equation}
\textup{Fix\ensuremath{\left(P_{A}P_{B}\right)}}=\left\{ z\in A\mid\left\Vert z-P_{B}\left(z\right)\right\Vert =dist\left(A,B\right)\right\} .
\end{equation}
Since  $\lim_{k\rightarrow\infty}\left(\widehat{Q}_{\mathcal{B},\tau^{\mathcal{B}},k,\beta}\widehat{Q}_{\mathcal{A},\tau^{\mathcal{A}},k,\alpha}\right)=P_{B}P_{A}$
and $\lim_{k\rightarrow\infty}\left(\widehat{Q}_{\mathcal{A},\tau^{\mathcal{A}},k,\alpha}\widehat{Q}_{\mathcal{B},\tau^{\mathcal{B}},k,\beta}\right)=P_{A}P_{B}$,
as follows from Theorem \bref{thm:UC-to-P(B)P(A)}, the previous lines
imply (\bref{eq:distance}) and (\bref{eq:distance 2}). 
\end{proof}
Our formulation and proof of Lemma \bref{lem:Lemma 2} below are inspired
by \cite[Lemma 2]{aharoni2018finding}, but a few  differences
exist between what is written in \cite{aharoni2018finding} and what
we do here, partly because the proof of \cite[Lemma 2]{aharoni2018finding}
contains a few minor gaps and other issues, and also because it is
possible to slightly generalize \cite[Lemma 2]{aharoni2018finding}
as we do here. 
\begin{lem}
\label{lem:Lemma 2} Let $K_{1}$ and $K_{2}$ be two nonempty, closed
and convex sets in a real Hilbert space $\mathcal{H}$, and let $P_{i}$
be the orthogonal projection onto $K_{i},\:i\in\left\{ 1,2\right\} $.
If $S$ is a nonempty and compact subset of $\mathcal{H}$ such that
$P_{2}P_{1}\left(S\right)=S$ and if $\textup{Fix}\left(P_{2}P_{1}\right)\neq\emptyset$,
then $S\subseteq\textup{Fix}\left(P_{2}P_{1}\right)$  and any $s\in S$
satisfies $d(s,K_{1})=\inf\left\{ d(y,K_{1}):y\in K_{2}\right\}$. Stated differently, if $\emptyset\neq S\subseteq\mathcal{H}$ is compact
such that $P_{2}P_{1}\left(S\right)=S$ and if $dist\left(K_{1},K_{2}\right)$
is attained, then any point of $S$ is both a fixed point of $P_{2}P_{1}$
and a point of $K_{2}$ which is nearest to $K_{1}$.
\end{lem}

\begin{proof}
Define 
\begin{equation}
S^{\prime}:=\left\{ s\in S\mid P_{2}P_{1}\left(s\right)=s\right\} .
\end{equation}

Since $\textup{Fix}\left(P_{2}P_{1}\right)\neq\emptyset$ there is some $s\in\mathcal{H}$ such that $s=P_2P_1(s)$, and because $P_{2}P_{1}\left(S\right)=S$, we have  $s=P_{2}P_{1}(s)\in S$. Hence $S'\neq\emptyset$.

In order to see that $S^{\prime}$ is also compact, we first show that it is
closed. Suppose that $s_{k}^{\prime}\in S^{\prime}$ for each $k\in\mathbb{N}$
and $\lim_{k\rightarrow\infty}s_{k}^{\prime}=s$ for some $s\in\mathcal{H}$.
By the definition of $S^{\prime}$ we have $P_{2}P_{1}\left(s_{k}^{\prime}\right)=s_{k}^{\prime}$
for each $k\in\mathbb{N}$. Because $s_{k}^{\prime}\in S^{\prime}\subseteq S$
for each $k\in\mathbb{N}$ and because $S$ is compact and hence closed,
we conclude that the limit $s$ is also in $S$. Since $P_{2}$ and
$P_{1}$ are continuous, so is their composition, and so $s=\lim_{k\rightarrow\infty}s_{k}^{\prime}=\lim_{k\rightarrow\infty}P_{2}P_{1}\left(s_{k}^{\prime}\right)=P_{2}P_{1}\left(s\right)$,
namely $s$ is a fixed point of $P_{2}P_{1}$. Since we already know
that $s\in S$, we have $s\in S^{\prime}$ by the definition of $S^{\prime}$,
and hence $S^{\prime}$ is closed. Thus $S^{\prime}$ is a closed
subset of the compact set $S$, and hence, as is well-known, this
implies that $S^{\prime}$ is compact.

Let $\hat{d}:=\sup\left\{ d\left(s,S^{\prime}\right)\mid s\in S\right\} $,
where $d\left(s,S^{\prime}\right):=\inf\left\{ \left\Vert s-s^{\prime}\right\Vert \mid s^{\prime}\in S^{\prime}\right\} $.
It is well-known and can easily be proved that the function $f(s):=d\left(s,E\right)$
is continuous (even  NE) whenever $E\neq\emptyset$ (in particular,
for $E:=S^{\prime}$), and so the compactness of $S$, and the Weierstrass
Extreme Value Theorem, imply that $f$ attains its maximum over $S$,
namely $f(y)=\hat{d}=d\left(y,S^{\prime}\right)$ for some $y\in S$.

Since $y\in S=P_{2}P_{1}\left(S\right)=\left\{ P_{2}P_{1}\left(x\right)\mid x\in S\right\} $,
there is some $x\in S$ such that $y=P_{2}P_{1}\left(x\right)$. This
fact, the continuity of the norm, the definition of $d\left(x,S^{\prime}\right)$,
the compactness of $S^{\prime}$ and the Weierstrass Extreme Value
Theorem, all imply that the infimum $\inf\left\{ \left\Vert x-s^{\prime}\right\Vert \mid s^{\prime}\in S^{\prime}\right\}$ is attained, namely  $d\left(x,S^{\prime}\right)=\left\Vert x-s^{\prime}\right\Vert $
for some $s^{\prime}\in S^{\prime}$. Therefore, 
\begin{align}
\left\Vert x-s^{\prime}\right\Vert =d\left(x,S^{\prime}\right)\leq\sup\left\{ d\left(s,S^{\prime}\right)\mid s\in S\right\} =\hat{d}\nonumber \\
=d\left(y,S^{\prime}\right)=\inf\left\{ \left\Vert y-s^{\prime\prime}\right\Vert \mid s^{\prime\prime}\in S^{\prime}\right\}  & \leq\left\Vert y-s^{\prime}\right\Vert .\label{eq:5.62}
\end{align}

We claim that $x\in S^{\prime}$. Indeed, suppose by way of contradiction
that this is not true. Since $s^{\prime}\in S^{\prime}\subseteq S=P_2P_1(S)\subseteq K_{2}$
and since any point in $S^{\prime}$ is a fixed point of $P_{2}P_{1}$,
we know from Theorem \bref{thm:cheney and goldstein Thm 2} that $s^{\prime}$
is a point in $K_{2}$ which is nearest to $K_{1}$. Thus $\left\Vert s^{\prime}-P_{1}\left(s^{\prime}\right)\right\Vert =dist\left(K_{2},K_{1}\right)$.
On the other hand, since $x\in S\subseteq K_{2}$ and since we assume
that $x\notin S^{\prime}$, it follows that $x$ is not a fixed point
of $P_{2}P_{1}$ (otherwise it would be in $S^{\prime}$ by the definition
of $S^{\prime}$, a contradiction). Therefore,  Theorem \bref{thm:cheney and goldstein Thm 2} implies
that $x$ cannot be a point in $K_{2}$ which is nearest to $K_{1}$,
and so $dist\left(K_{2},K_{1}\right)<\left\Vert x-P_{1}\left(x\right)\right\Vert $.
We conclude from the previous lines that $\left\Vert s^{\prime}-P_{1}\left(s^{\prime}\right)\right\Vert =dist\left(K_{2},K_{1}\right)<\left\Vert x-P_{1}\left(x\right)\right\Vert $.

Consequently, from the necessary condition for equality in the nonexpansivness
property of an orthogonal projection \cite[Theorem 3]{cheney1959proximity}
(see the second part of Proposition \bref{prop:P_C=00003D00003DNE}),
we conclude that $\left\Vert P_{1}\left(x\right)-P_{1}\left(s^{\prime}\right)\right\Vert <\left\Vert x-s^{\prime}\right\Vert $.
This inequality, and the facts that $P_{2}$ is NE , that $s^{\prime}=P_{2}P_{1}\left(s^{\prime}\right)$
and that $y=P_{2}P_{1}\left(x\right)$, all imply that 
\begin{equation}
\left\Vert y-s^{\prime}\right\Vert =\left\Vert P_{2}P_{1}\left(x\right)-P_{2}P_{1}\left(s^{\prime}\right)\right\Vert \leq\left\Vert P_{1}\left(x\right)-P_{1}\left(s^{\prime}\right)\right\Vert <\left\Vert x-s^{\prime}\right\Vert ,\label{eq:5.63}
\end{equation}
and this contradicts (\bref{eq:5.62}). Thus, the initial assumption
that $x\notin S^{\prime}$ cannot be true, namely $x\in S^{\prime}$,
and hence $x$ is a fixed point of $P_{2}P_{1}$. Therefore, $x=P_{2}P_{1}\left(x\right)=y$ and hence, from the fact that $x\in S^{\prime}$, we conclude that 
$y\in S^{\prime}$. Thus, $d\left(y,S^{\prime}\right)=0$ and so $\hat{d}=0$.
Since $0\leq d\left(s,S^{\prime}\right)\leq\hat{d}=0$ for all $s\in S$
by the definition of $\hat{d}$, it follows that each $s\in S$ satisfies
$d\left(s,S^{\prime}\right)=0$. But $S^{\prime}$ is closed, as we
showed earlier, and hence every $s\in S$ is actually in $S^{\prime}$.
Therefore, $S\subseteq S^{\prime}$, and since obviously $S^{\prime}\subseteq S$,
we have $S=S^{\prime}$, that is, every point in $S$ is a fixed point
of $P_{2}P_{1}$, as claimed. Finally, from   Theorem \bref{thm:cheney and goldstein Thm 2}
we know that $\textup{Fix}\left(P_{2}P_{1}\right)\neq\emptyset$ if
and only if $dist\left(K_{1},K_{2}\right)$ is attained, and, moreover,
that $s$ is a fixed point of $P_{2}P_{1}$ if and only if $s$ is a
point of $K_{2}$ which is nearest to $K_{1}.$

\end{proof}
\begin{rem}
An alternative proof of Lemma \bref{lem:Lemma 2} was suggested by
one of the referees. This proof is based on the fact that $P_{2}P_{1}$
is averaged (since both $P_{1}$ and $P_{2}$ are firmly nonexpansive
\cite[Theorem 2.2.21 (iii)]{Ceg-book} and hence, as follows from
\cite[Lemma 2.3]{Byrne2004jour}, both $P_{1}$ and $P_{2}$ are
averaged and therefore, as follows from either \cite[Proposition 2.1]{Byrne2004jour}
or \cite[Proposition 4.44]{BC17}, the composition $P_{2}P_{1}$
is averaged) and hence, using the fact that $\textup{Fix}\left(P_{2}P_{1}\right)$
is nonempty, $P_{2}P_{1}$ is actually strongly quasi-nonexpansive
(since $P_{2}P_{1}$ is a relaxed firmly nonexpansive operator by
\cite[Corollary 2.2.17)]{Ceg-book}, and therefore strongly quasi-nonexpansive
as follows from \cite[Corollary 2.2.15]{Ceg-book}). 
\end{rem}

Now we are able to present the A-S-HLWB algorithm:\\ 
\begin{algorithm}[H]
\noindent \textbf{Input: } Two positive integers $I$ and $J,$ two
families\textbf{ }$\left\{ A_{i}\right\} _{i=1}^{I}$ and $\left\{ B_{j}\right\} _{j=1}^{J}$
of closed and convex subsets in a Euclidean space \textit{$\mathcal{H}$,}
two sequences $\tau^{\mathcal{A}}=\left\{ \tau_{t}^{\mathcal{A}}\right\} _{t=0}^{\infty}$
and $\tau^{\mathcal{B}}=\left\{ \tau_{t}^{\mathcal{B}}\right\} _{t=0}^{\infty}$ of steering
parameters, two positive weight vectors \textcolor{black}{$\alpha=(\alpha_{i})_{i=1}^{I}$
and $\beta=(\beta_{j})_{j=1}^{J}$}.

\setlength{\parindent}{0pt}
\noindent\textbf{Initialization:} Choose an arbitrary starting point
\textit{$x^{0}\in\mathcal{H}$}.

\noindent \textbf{Iterative Step}: Given $x^{k}$, $k\in\mathbb{N}\cup\{0\}$, find the next iterate
$x^{k+1}$, as follows:

\noindent (a) If $k$ is even, then define

\noindent 
\begin{equation}\label{eq:xk_even}
x^{k+1}:=\widehat{Q}_{\mathcal{A},\tau^{\mathcal{A}},\frac{k}{2},\alpha}(x^{k}).
\end{equation}

\noindent (b) If $k$ is odd, then define 

\noindent 

\noindent 
\begin{equation}\label{eq:xk_odd}
x^{k+1}:=\widehat{Q}_{\mathcal{B},\tau^{\mathcal{B}},\frac{k-1}{2},\beta}(x^{k}).
\end{equation}

\caption{The new alternating simultaneous HLWB (A-S-HLWB) algorithm}
\label{alg:our-alg} 
\end{algorithm}

As an illustration, given an initialization point $x^{0}$, we write below explicitly the first six iterations out of the infinite sequence, generated by the A-S-HLWB algorithm (observe that the number of sweeps in both the even and the odd iterations increases as the iterations increase):

\small 
\begin{equation}
 x^{1}=\widehat{Q}_{\mathcal{A},\tau^{\mathcal{A}},0,\alpha}\left(x^{0}\right)=\left(\tau_{0}^{\mathcal{A}}\textup{Id}+\left(1-\tau_{0}^{\mathcal{A}}\right)\sum_{i=1}^{I}\alpha_{i}P_{A_{i}}\right)\left(x^{0}\right).
\end{equation}

\begin{equation}
 x^{2}=\widehat{Q}_{\mathcal{B},\tau^{\mathcal{B}},0,\beta}\left(x^{1}\right)=\left(\tau_{0}^{\mathcal{B}}\textup{Id}+\left(1-\tau_{0}^{\mathcal{B}}\right)\sum_{j=1}^{J}\beta_{j}P_{B_{j}}\right)\left(x^{1}\right).
\end{equation}

\begin{equation}
 x^{3}=\widehat{Q}_{\mathcal{A},\tau^{\mathcal{A}},1,\alpha}\left(x^{2}\right)=\left(\tau_{1}^{\mathcal{A}}\textup{Id}+\left(1-\tau_{1}^{\mathcal{A}}\right)\sum_{i=1}^{I}\alpha_{i}P_{A_{i}}\right)\left(\tau_{0}^{\mathcal{A}}\textup{Id}+\left(1-\tau_{0}^{\mathcal{A}}\right)\sum_{i=1}^{I}\alpha_{i}P_{A_{i}}\right)\left(x^{2}\right).
\end{equation}

\begin{equation}
 x^{4}=\widehat{Q}_{\mathcal{B},\tau^{\mathcal{B}},1,\beta}\left(x^{3}\right)=\left(\tau_{1}^{\mathcal{B}}\textup{Id}+\left(1-\tau_{1}^{\mathcal{B}}\right)\sum_{\ell=1}^{J}\beta_{j}P_{B_{j}}\right)\left(\tau_{0}^{\mathcal{B}}\textup{Id}+\left(1-\tau_{0}^{\mathcal{B}}\right)\sum_{\ell=1}^{J}\beta_{j}P_{B_{j}}\right)\left(x^{3}\right).
\end{equation}

\begin{equation}
 x^{5}=\widehat{Q}_{\mathcal{A},\tau^{\mathcal{A}},2,\alpha}\left(x^{4}\right)=\prod_{t=0}^{2}\left(\tau_{t}^{\mathcal{A}}\textup{Id}+\left(1-\tau_{t}^{\mathcal{A}}\right)\sum_{i=1}^{I}\alpha_{i}P_{A_{i}}\right)\left(x^{4}\right).
\end{equation}

\begin{equation}
 x^{6}=\widehat{Q}_{\mathcal{B},\tau^{\mathcal{B}},2,\beta}\left(x^{5}\right)=\prod_{t=0}^{2}\left(\tau_{t}^{\mathcal{B}}\textup{Id}+\left(1-\tau_{t}^{\mathcal{B}}\right)\sum_{j=1}^{J}\beta_{j}P_{B_{j}}\right)\left(x^{5}\right).
\end{equation}
\normalsize

Hence, in the computation of $x^{1}$ one applies a projection operator
(from the pool $\left\{ P_{A_{1}},\ldots,P_{A_{I}}\right\} $) $I$
times on $x^{0}$. In the computation of $x^{2}$ one first computes
$x^{1}$ by applying a projection operator $I$ times and then applies
a projector operator (from the pool $\left\{ P_{B_{1}},\ldots,P_{B_{J}}\right\} $)
$J$ times on $x^{1}$, and so in total one applies a projector operator
$I+J$ times. In the computation of $x^{3}$ one applies a projection
operator (from the pool $\left\{ P_{A_{1}},\ldots,P_{A_{I}}\right\} $)
$I$ times on $x^{2}$ and then another $I$ times (again from the
pool $\left\{ P_{A_{1}},\ldots,P_{A_{I}}\right\} $) on $\left(\tau_{0}^{\mathcal{A}}\textup{Id}+\left(1-\tau_{0}^{\mathcal{A}}\right)\sum_{i=1}^{I}\alpha_{i}P_{A_{i}}\right)\left(x^{2}\right)$, and since for computing $x^2$ one applies a projection operator $I+J$ times as we showed earlier, in total one applies a projection operator (from the pools
$\left\{ P_{A_{1}},\ldots,P_{A_{I}}\right\} $ and $\left\{ P_{B_{1}},\ldots,P_{B_{J}}\right\}$)  $I+J+2I$ times during the computation of $x^3$.  Similarly, in the computation of $x^{4}$ one applies
a projection operator (from the pools $\left\{ P_{A_{1}},\ldots,P_{A_{I}}\right\} $
and $\left\{ P_{B_{1}},\ldots,P_{B_{J}}\right\}$)  $I+J+2I+2J$ times.
In general, in the computation of $x^{2r}$ for some natural number
$r$ one applies a projection operator $I+J+2I+2J+\ldots+rI+rJ=(I+J)(1+2+\ldots+r)=(I+J)r(r+1)/2$
times.

\section{Convergence of the A-S-HLWB algorithm\label{sec:Convergence}}

In this section we present our main convergence result, namely Theorem
\bref{thm:Main} below. Its proof is inspired by the first part of
the proof of \cite[Theorem 1]{aharoni2018finding}. Before formulating
the theorem, we need a lemma. 
\begin{lem}
\label{lem:xk_in_Ball} If there is some $\rho>0$ such $A_{i},B_{j}\subseteq B\left[0,\rho\right]$
for all $i\in\left\{ 1,2,\ldots,I\right\} $ and $j\in\left\{ 1,2,\ldots,J\right\} $,
 if $x^{0}\in B[0,\rho]$ and if $\left\{ x^{k}\right\} _{k=0}^{\infty}$
is generated using Algorithm \bref{alg:our-alg}, then $x^{k}\in B[0,\rho]$ for all
$k\in\mathbb{N\cup\left\{ \mathrm{0}\right\} }$. 
\end{lem}

\begin{proof}
In order to prove this assertion, we first show by induction that
if $x\in B[0,\rho],$ then  $\widehat{Q}_{\mathcal{A},\tau^{\mathcal{A}},k,\alpha}(x)\in B[0,\rho]$
for every $k\in\mathbb{N\cup\left\{ \mathrm{0}\right\} }$. Indeed,
recall that  $\widehat{Q}_{\mathcal{A},\tau^{\mathcal{A}},k,\alpha}(x)=\prod_{t=0}^{k}\widehat{M}_{\mathcal{A},\tau_{t}^{\mathcal{A}},\alpha}(x).$
Since $x\in B[0,\rho]$ and $A_{i}$ is contained in $B\left[0,\rho\right]$
for all $i\in I$, we have  $P_{A_{i}}\left(x\right)\in A_{i}\subseteq B\left[0,\rho\right]$
for all $i\in\left\{ 1,2,\ldots,I\right\} $ and so, by the convexity
of $B\left[0,\rho\right]$, also $\sum_{i=1}^{I}\alpha_{i}P_{A_{i}}\left(x\right)\in B\left[0,\rho\right]$.
Again by the convexity of $B\left[0,\rho\right]$ and the fact that
$\tau_{0}^{\mathcal{A}}\in(0,1)$, also  $\tau_{0}^{\mathcal{\mathcal{A}}}x+\left(1-\tau_{0}^{\mathcal{\mathcal{A}}}\right)\sum_{i=1}^{I}\alpha_{i}P_{A_{i}}\left(x\right)$
is in $B\left[0,\rho\right]$, namely $\widehat{Q}_{\mathcal{A},\tau^{\mathcal{A}},0,\alpha}(x)=\widehat{M}_{\mathcal{A},\tau_{0}^{\mathcal{A}},\alpha}(x)\in B\left[0,\rho\right]$.
Assume now that $k\in\mathbb{N}$ and that the induction hypothesis
holds for all $q\in\left\{ 0,1,\ldots,k-1\right\} $, that is, $\widehat{Q}_{\mathcal{A},\tau^{\mathcal{A}},q,\alpha}(x)\in B[0,\rho]$
for all $q\in\left\{ 0,1,\ldots,k-1\right\} $. Since $\widehat{M}_{\mathcal{A},\tau_{k}^{\mathcal{\mathcal{A}}},\alpha}(y)\in B[0,\rho]$
whenever $y\in B[0,\rho]$ by a similar argument to the one used in
previous lines (we use the same argument which we used in order to
show that  $\widehat{M}_{\mathcal{A},\tau_{0}^{\mathcal{A}},\alpha}(x)\in B\left[0,\rho\right]$
whenever $x\in B[0,\rho]$, but now with  $\widehat{M}_{\mathcal{A},\tau_{k}^{\mathcal{\mathcal{A}}},\alpha}(y)$
and $y$ instead of  $\widehat{M}_{\mathcal{A},\tau_{0}^{\mathcal{A}},\alpha}(x)$
and $x$, respectively), we see that for $y:=\widehat{Q}_{\mathcal{A},\tau^{\mathcal{A}},k-1,\alpha}(x)$,
we have $\widehat{Q}_{\mathcal{A},\tau^{\mathcal{A}},k,\alpha}(x)=\widehat{M}_{\mathcal{A},\tau_{k}^{\mathcal{\mathcal{A}}},\alpha}(\widehat{Q}_{\mathcal{A},\tau^{\mathcal{A}},k-1,\alpha}(x))=\widehat{M}_{\mathcal{A},\tau_{k}^{\mathcal{\mathcal{A}}},\alpha}(y)\in B[0,\rho]$,
since by the induction hypothesis for $k-1$ we have $y\in B[0,\rho]$.
Hence the induction hypothesis holds for $k$ as well, and so  $\widehat{Q}_{\mathcal{A},\tau^{\mathcal{A}},k,\alpha}(x)\in B[0,\rho]$
for every $k\in\mathbb{N\cup\left\{ \mathrm{0}\right\} }$, as required.
Similarly, if $x\in B[0,\rho]$, then $\widehat{Q}_{\mathcal{B},\tau^{\mathcal{B}},k,\beta}(x)\in B[0,\rho]$
for every $k\in\mathbb{N\cup\left\{ \mathrm{0}\right\} }$.

Finally, in order to see that $x^{k}\in B[0,\rho]$ for all $k\in\mathbb{N\cup\left\{ \mathrm{0}\right\} }$,
we apply induction on $k$. By our assumption $x^{0}\in B[0,\rho].$
Assume that the induction hypothesis holds for all $q\in\left\{ 0,1,\ldots,k\right\} ,$
namely, that $x^{q}\in B[0,\rho]$ for all $q\in\left\{ 0,1,\ldots,k\right\} .$
If $k$ is even, then, according to (\bref{eq:xk_even}), one has $x^{k+1}=\widehat{Q}_{\mathcal{A},\tau^{\mathcal{A}},k/2,\alpha}(x^{k})$,
and so by previous lines and the induction hypothesis (that $x:=x^{k}\in B[0,\rho]$)
it follows that  $x^{k+1}=\widehat{Q}_{\mathcal{A},\tau^{\mathcal{A}},k/2,\alpha}(x)\in B[0,\rho],$
as well. If $k$ is odd, then, according to (\bref{eq:xk_odd}), one
has $x^{k+1}=\widehat{Q}_{\mathcal{B},\tau^{\mathcal{B}},(k-1)/2,\beta}(x^{k})$, and
so by previous lines and the induction hypothesis (with $x:=x^{k}\in B[0,\rho]$ and $q:=k/2$)
it follows that  $x^{k+1}=\widehat{Q}_{\mathcal{B},\tau^{\mathcal{B}},(k-1)/2,\beta}(x)\in B[0,\rho],$
as well. Hence the induction hypothesis holds for $k+1.$ Therefore
indeed $x^{k}\in B[0,\rho]$ for every nonnegative integer $k$, as
required. 
\end{proof}

\begin{thm}
\label{thm:Main}Given are two families of closed and  convex sets $\left\{ A_{i}\right\} _{i=1}^{I}$
and $\left\{ B_{j}\right\} _{j=1}^{J}$ in a Euclidean space $\mathcal{H}$
for some positive integers $I$ and $J,$ two sequences of steering
parameters $\tau^{\mathcal{A}}=\left\{ \tau_{k}^{\mathcal{A}}\right\} _{k=1}^{\infty}$
and $\tau^{\mathcal{B}}=\left\{ \tau_{k}^{\mathcal{B}}\right\} _{k=1}^{\infty}$, and
two positive weight vectors \textcolor{black}{$\alpha=(\alpha_{i})_{i=1}^{I}$
and $\beta=(\beta_{j})_{j=1}^{J}$}. Assume that $A:=\cap_{i=1}^{I}A_{i}\neq\emptyset$
and $B:=\cap_{j=1}^{J}B_{j}\neq\emptyset$ but $A\cap B=\emptyset$.
Assume also that there is some $\rho>0$ such that $A_{i},B_{j}\subseteq B\left[0,\rho\right]$
for all $i\in\left\{ 1,2,\ldots,I\right\} $ and $j\in\left\{ 1,2,\ldots,J\right\} $. 
Further assume that there is a unique best approximation pair relative
to $\left(A,B\right)$. Let \textup{$\left\{ x^{k}\right\} _{k=0}^{\infty}$}
be a sequence generated by Algorithm \bref{alg:our-alg}, where we
assume that $x^{0}\in B[0,\rho]$. Then 
\begin{equation}
\lim_{k\rightarrow\infty}\left\Vert x^{k+1}-x^{k}\right\Vert =dist\left(A,B\right),
\end{equation}
and moreover, the odd subsequence $\left\{ x^{2k+1}\right\} _{k=0}^{\infty}$
converges to a point $a\in A$, the even subsequence $\left\{ x^{2k}\right\} _{k=0}^{\infty}$
converges to a point $b\in B$, and $\left(a,b\right)$ is a best
approximation pair relative to $\left(A,B\right)$. In particular,
the conclusion of the theorem holds when all the sets $A_{i}$ and
$B_{j}$, $i\in\left\{ 1,2,\ldots,I\right\} $ and $j\in\left\{ 1,2,\ldots,J\right\} $,
are strictly convex. 
\end{thm}

\begin{proof}
Since $A_{i}$ and $B_{j}$ for $i\in\left\{ 1,2,\ldots,I\right\} $
and $j\in\left\{ 1,2,\ldots,J\right\} $ are closed subsets of the
closed (hence compact) ball $B\left[0,\rho\right]$, they are compact
and hence also their intersections $A=\cap_{i=1}^{I}A_{i}$ and $B=\cap_{j=1}^{J}B_{j}$
are. Thus, $dist\left(A,B\right)$ is attained by the continuity of
the norm and the Weierstrass Extreme Value Theorem. Therefore, the
conditions needed in Theorem \bref{thm:cheney and goldstein Thm 2}
are satisfied, and hence $\textup{Fix}\left(P_{B}P_{A}\right)$ is
nonempty. Since $A_{i},B_{j}\subseteq B\left[0,\rho\right]$ for all
$i\in\left\{ 1,2,\ldots,I\right\} $ and $j\in\left\{ 1,2,\ldots,J\right\} $,
by Theorem \bref{thm:UC-to-P(B)P(A)} we have $\lim_{k\rightarrow\infty}\left(\widehat{Q}_{\mathcal{B},\tau^{\mathcal{B}},k,\beta}\widehat{Q}_{\mathcal{A},\tau^{\mathcal{A}},k,\alpha}\right)=P_{B}P_{A}$.
 In addition, Lemma \bref{lem:xk_in_Ball} ensures that $x^{k}\in B[0,\rho]$
for all nonnegative integer $k$.

Let $S$ be the set of accumulation points of $\left\{ x^{2r}\right\} _{r=0}^{\infty}$.
By the Bolzano--Weierstrass theorem $S\neq\emptyset$. Moreover, since, as is well known, the set of accumulation points of a sequence is closed and since $S$ is contained in the compact set $B\left[0,\rho\right]$, we conclude that $S$ is compact too.

We claim that $P_{B}P_{A}\left(S\right)=S$. Indeed, let $s\in S$
be any point. Then there is a subsequence $\left\{ x^{2r_{k}}\right\} _{k=0}^{\infty}$
such that $s=\lim_{k\rightarrow\infty}x^{2r_{k}}$, and so for a given
$\varepsilon_{1}>0$ there is some natural number $k_{0}$ such that
$\left\Vert s-x^{2r_{k}}\right\Vert \leq\varepsilon_{1}$ for all
natural numbers $k>k_{0}$. Due to Theorem \bref{thm:UC-to-P(B)P(A)}
and since the sequence $\left\{ x^{k}\right\} _{k=0}^{\infty}$ is
contained in $B\left[0,\rho\right]$ by Lemma \bref{lem:xk_in_Ball},
for a given $\varepsilon_{2}>0$ there is a natural number $r_{0}$
such that for all natural numbers $r>r_{0}$ one has $\left\Vert P_{B}P_{A}\left(y\right)-\widehat{Q}_{\mathcal{B},\tau^{\mathcal{B}},r,\beta}\widehat{Q}_{\mathcal{A},\tau^{\mathcal{A}},r,\alpha}\left(y\right)\right\Vert \leq\varepsilon_{2}$
for all $y\in B\left[0,\rho\right]$ and, in particular, for $y:=x^{2r}$.
Let $\varepsilon>0$ and define $\varepsilon_{1}:=\dfrac{\varepsilon}{2}$
and $\varepsilon_{2}:=\dfrac{\varepsilon}{2}$. For $\varepsilon_{1}$
and $\varepsilon_{2}$ we can associate the natural numbers $k_{0}$
and $r_{0}$, respectively, mentioned above, and hence, if $k$ is
a natural number satisfying $k>k_{0}$ and $r_{k}>r_{0}$ , then,
by the nonexpansivity of $P_{B}P_{A}$ and by the construction of the sequence $\left\{ x^{k}\right\} _{k=0}^{\infty}$ (using Algorithm \bref{alg:our-alg}), we have 
\begin{align}
 & \left\Vert P_{B}P_{A}\left(s\right)-x^{2r_{k}+2}\right\Vert =\left\Vert P_{B}P_{A}\left(s\right)-\widehat{Q}_{\mathcal{B},\tau^{\mathcal{B}},r_{k},\beta}\widehat{Q}_{\mathcal{A},\tau^{\mathcal{A}},r_{k},\alpha}\left(x^{2r_{k}}\right)\right\Vert \nonumber \\
 & \leq\left\Vert P_{B}P_{A}\left(s\right)-P_{B}P_{A}\left(x^{2r_{k}}\right)\right\Vert +\left\Vert P_{B}P_{A}\left(x^{2r_{k}}\right)-\widehat{Q}_{\mathcal{B},\tau^{\mathcal{B}},r_{k},\beta}\widehat{Q}_{\mathcal{A},\tau^{\mathcal{A}},r_{k},\alpha}\left(x^{2r_{k}}\right)\right\Vert \nonumber \\
 & \leq\left\Vert s-x^{2r_{k}}\right\Vert +\left\Vert P_{B}P_{A}\left(x^{2r_{k}}\right)-\widehat{Q}_{\mathcal{B},\tau^{\mathcal{B}},r_{k},\beta}\widehat{Q}_{\mathcal{A},\tau^{\mathcal{A}},r_{k},\alpha}\left(x^{2r_{k}}\right)\right\Vert \nonumber \\
 & \leq\varepsilon_1+\varepsilon_2=\varepsilon.
\end{align}

This implies that $P_{B}P_{A}\left(s\right)$ is also an accumulation
point of $\left\{ x^{2r}\right\} _{r=0}^{\infty}$. Since $s$ was
an arbitrary point in $S$, we conclude that,  $P_{B}P_{A}\left(S\right)\subseteq S$.

On the other hand, suppose that $s\in S$. Then $s$ is the limit
of a subsequence of $\left\{ x^{2r}\right\} _{r=0}^{\infty}$, namely
$s=\lim_{\ell\rightarrow\infty}x^{2r_{\ell}}$ for some subsequence
$\left\{ x^{2r_{\ell}}\right\} _{\ell=0}^{\infty}$. Since $\left\{ x^{2r_{\ell}-2}\right\} _{\ell=1}^{\infty}$
is a subsequence of the sequence $\left\{ x^{t}\right\} _{t=1}^{\infty}$
which is contained in the compact set $B\left[0,\rho\right]$, it
follows from the compactness of $B\left[0,\rho\right]$ that $\left\{ x^{2r_{\ell}-2}\right\} _{\ell=1}^{\infty}$
has a limit point $s^{\prime}\in B\left[0,\rho\right]$. Hence $s^{\prime}=\lim_{k\rightarrow\infty}x^{2r_{\ell_{k}}-2}$
for some subsequence $\left\{ x^{2r_{\ell_{k}}-2}\right\} _{k=1}^{\infty}$
of $\left\{ x^{2r_{\ell}-2}\right\} _{\ell=1}^{\infty}$ , and therefore,
given $\varepsilon_{1}>0$, there is some $k_{0}\in\mathbb{N}$ such
that $\left\Vert s^{\prime}-x^{2r_{\ell_{k}}-2}\right\Vert \leq\varepsilon_{1}$
for all $k>k_{0}$. In addition, since $s^{\prime}$ is the limit
of a subsequence of the sequence $\left\{ x^{2t}\right\} _{t=0}^{\infty}$,
the definition of $S$ implies that $s^{\prime}\in S$. Due to Theorem
\bref{thm:UC-to-P(B)P(A)} and since the sequence $\left\{ x^{t}\right\} _{t=0}^{\infty}$
is contained in $B\left[0,\rho\right]$ by Lemma \bref{lem:xk_in_Ball},
 for a given $\varepsilon_{2}>0$ there is a natural number $\ell_{0}$
such that for all natural numbers $\ell>\ell_{0}$ and for all $y\in B\left[0,\rho\right]$,
one has $\left\Vert P_{B}P_{A}\left(y\right)-\widehat{Q}_{\mathcal{B},\tau^{\mathcal{B}},r_{\ell}-1,\beta}\widehat{Q}_{\mathcal{A},\tau^{\mathcal{A}},r_{\ell}-1,\alpha}\left(y\right)\right\Vert \leq\varepsilon_{2}$.
This inequality is true, in particular, for $y:=x^{2r_{\ell}-2}$.
Let $\varepsilon>0$ and define $\varepsilon_{1}:=\dfrac{\varepsilon}{2}$
and $\varepsilon_{2}:=\dfrac{\varepsilon}{2}$. For $\varepsilon_{1}$
and $\varepsilon_{2}$ we can associate the natural numbers $k_{0}$
and $\ell_{0}$ mentioned above, and hence, if $k$ is a natural number
satisfying $k>k_{0}$ and $\ell_{k}>\ell_{0}$, then, by the nonexpansivity
of $P_{B}P_{A}$, we have 

\footnotesize
\begin{align}
 & \left\Vert P_{B}P_{A}\left(s^{\prime}\right)-x^{2r_{\ell_{k}}}\right\Vert=\left\Vert P_{B}P_{A}\left(s^{\prime}\right)-\widehat{Q}_{\mathcal{B},\tau^{\mathcal{B}},r_{\ell_{k}}-1,\beta}\widehat{Q}_{\mathcal{A},\tau^{\mathcal{A}},r_{\ell_{k}}-1,\alpha}\left(x^{2r_{\ell_{k}}-2}\right)\right\Vert \nonumber \\
 & \leq\left\Vert P_{B}P_{A}\left(s^{\prime}\right)-P_{B}P_{A}\left(x^{2r_{\ell_{k}}-2}\right)\right\Vert+\left\Vert P_{B}P_{A}\left(x^{2r_{\ell_{k}}-2}\right)-\widehat{Q}_{\mathcal{B},\tau^{\mathcal{B}},r_{\ell_{k}}-1,\beta}\widehat{Q}_{\mathcal{A},\tau^{\mathcal{A}},r_{\ell_{k}}-1,\alpha}\left(x^{2r_{\ell_{k}}-2}\right)\right\Vert \nonumber \\
&  \leq\left\Vert s^{\prime}-x^{2r_{\ell_{k}}-2}\right\Vert +\left\Vert P_{B}P_{A}\left(x^{2r_{\ell_{k}}-2}\right)-\widehat{Q}_{\mathcal{B},\tau^{\mathcal{B}},r_{\ell_{k}}-1,\beta}\widehat{Q}_{\mathcal{A},\tau^{\mathcal{A}},r_{\ell_{k}}-1,\alpha}\left(x^{2r_{\ell_{k}}-2}\right)\right\Vert \nonumber \\
&  \leq\varepsilon_1+\varepsilon_2=\varepsilon.
\end{align}
\normalsize

Therefore, $P_{B}P_{A}\left(s'\right)$ is an accumulation point of
the subsequence $\left\{ x^{2r_{\ell}}\right\} _{\ell=1}^{\infty}$,
and since we know that this subsequence also converges to $s$ by
the choice of $\left\{ x^{2r_{\ell}}\right\} _{\ell=1}^{\infty}$,
it follows that $s=P_{B}P_{A}\left(s'\right)\in P_{B}P_{A}\left(S\right)$.
Since $s$ was an arbitrary point in $S$, we conclude that $S\subseteq P_{B}P_{A}\left(S\right)$,
and since we already know that $P_{B}P_{A}\left(S\right)\subseteq S$,
we conclude that $P_{B}P_{A}\left(S\right)=S$. By Lemma \bref{lem:Lemma 2},
$S$ consists of points of $B$ nearest to $A$.

Similarly, the set $\tilde{S}$ of all accumulation points of the
odd subsequence $\left\{ x^{2k+1}\right\} _{k=0}^{\infty}$ is contained
in $A$ and satisfies $\tilde{S}=P_{A}P_{B}\left(\tilde{S}\right)$,
and so $\tilde{S}$ consists of points of $A$ nearest to $B$ by
Lemma \bref{lem:Lemma 2}. Let $\tilde{a}\in\tilde{S}$. As explained
before, Lemma \bref{lem:Lemma 2} implies that $\left(\tilde{a},P_{B}\left(\tilde{a}\right)\right)$
is a best approximation pair relative to $\left(A,B\right)$. Since
we assume that there is a unique best approximation pair $\left(a,b\right)\in A\times B$
relative to $\left(A,B\right)$, we conclude, in particular, that
$\tilde{a}=a$. Hence, $\tilde{S}=\left\{ a\right\} $, and similarly
$S=\left\{ b\right\} $. Therefore, $\left\{ x^{2k}\right\} _{k=0}^{\infty}$
has a unique accumulation point, namely it converges to $b$, and
$\left\{ x^{2k+1}\right\} _{k=0}^{\infty}$ has a unique accumulation
point, namely it converges to $a$. Now the continuity of the norm
and the fact that $\left(a,b\right)$ is a best approximation pair
relative to $\left(A,B\right)$ imply that $\lim_{k\rightarrow\infty}\left\Vert x^{2k}-x^{2k+1}\right\Vert =\left\Vert b-a\right\Vert =dist\left(A,B\right)$,
as required. Finally, Proposition \bref{prop:Proposition 6} implies
that the conclusion of the theorem holds when all the sets $A_{i}$
and $B_{j}$, $i\in\left\{ 1,2,\ldots,I\right\} $ and $j\in\left\{ 1,2,\ldots,J\right\} $,
are strictly convex. 
\end{proof}

\section{Concluding remarks and lines for further investigation\label{sec:Concluding-remarks}}

We would like to finish this paper with a few remarks, some of them
point to possible lines for further investigation, First, as noted
by one of the referees, it will be interesting to investigate whether
our main convergence result (Theorem \bref{thm:Main}) holds also when
there is more than one best approximation pair. In fact, even the
case when there is a unique best approximation pair deserves more
attention, namely to show that this case holds in a much richer setting
than when all the subsets $\left\{ A_{i}\right\} _{i=1}^{I}$ and
$\left\{ B_{j}\right\} _{j=1}^{J}$ are strictly convex: here we had
a progress, which can be found in \cite{ReemCensor2024prep}. It will also be interesting
to extend our result to infinite-dimensional spaces and to spaces
which are not Hilbert spaces (under the assumption that the projection
on a nonempty, closed and convex subset is unique, as in the case
of uniformly convex Banach spaces), where here one might replace the compactness
assumption on the subsets $\left\{ A_{i}\right\} _{i=1}^{I}$ and
$\left\{ B_{j}\right\} _{j=1}^{J}$ by their boundedness and the strong
convergence by weak convergence. Another interesting line for further
investigation was suggested by one of the referees: to check whether
in our Euclidean setting the compactness assumption on the subsets
$\left\{ A_{i}\right\} _{i=1}^{I}$ and $\left\{ B_{j}\right\} _{j=1}^{J}$
can be replaced by the assumption that $\textup{Fix}\left(P_{B}P_{A}\right)\neq\emptyset$;
in view of Theorem \bref{thm:cheney and goldstein Thm 2}, this latter assumption
is equivalent to assuming that there is at least one best approximation
pair relative to $(A,B)$ without further assumptions on the subsets
$\left\{ A_{i}\right\} _{i=1}^{I}$ and $\left\{ B_{j}\right\} _{j=1}^{J}$
(with the exception that they are nonempty, closed and convex); in
this connection we note that we use the compactness assumption several
times in our derivation, such as in Lemma \bref{lem:UC-to-Pc}, Theorem
\bref{thm:UC-to-P(B)P(A)}, Theorem \bref{thm:adapted cheney and goldstein}
and Lemma \bref{lem:xk_in_Ball}, to name a few places, and so a careful
analysis and possibly a significant change to our proofs will be required.
Finally, it will be helpful to provide explicit estimates regarding
the convergence rate of the odd and even subsequences. 

\subsubsection*{Acknowledgments}

We thank the referees for valuable remarks which helped to improve
the paper. D.R. thanks Zilin Jiang for helpful discussion. This research
is supported by the ISF-NSFC joint research plan Grant Number 2874/19.
The work of Y.C. was also supported by U. S. National Institutes of
Health grant R01CA266467.

\bibliographystyle{acm}

\end{document}